\documentclass[12pt]{amsart}

\makeindex

\usepackage[all, cmtip]{xy}
\usepackage{amssymb}
\usepackage{amsmath}
\usepackage{amsthm}
\usepackage{amscd}
\usepackage{amsfonts}
\usepackage{amssymb}
\usepackage[pdftex]{graphicx}
\usepackage{multirow}
\usepackage[usenames,dvipsnames,svgnames,table]{xcolor}
\usepackage{graphicx}
\usepackage{enumerate}

\usepackage[bookmarks=true,bookmarksnumbered=true,
 pdftitle={},
 pdfsubject={},
 pdfauthor={},
 pdfkeywords={TeX, dvipdfmx, hyperref, color,},
 colorlinks=true, linkcolor=RoyalBlue, citecolor=Blue]
 {hyperref}

\theoremstyle{definition}

\theoremstyle{remark}

\theoremstyle{corollary}

\newtheorem*{conjecturereg}{Regularity Conjecture}

\theoremstyle{theorem}
\newtheorem{theoremalpha}{Theorem}
\theoremstyle{corollary}

\newtheorem{theorem}{Theorem}[section]
\newtheorem{lemma}[theorem]{Lemma}
\newtheorem{proposition}[theorem]{Proposition}

\theoremstyle{corollary}
\newtheorem{corollary}[theorem]{Corollary}

\theoremstyle{definition}

\theoremstyle{remark}
\newtheorem{remark}[theorem]{Remark}

\numberwithin{equation}{section}

\newcommand{\Z}{\mathbb{Z}}

\def\P{\mathbb{P}}

\def\Pic{\operatorname{Pic}}

\def\reg{\operatorname{reg}}

\def\codim{\operatorname{codim}}
\def\rank{\operatorname{rank}}
\def\Bs{\operatorname{Bs}}

\newcommand{\suchthat}{\;\ifnum\currentgrouptype=16 \middle\fi|\;}

\input xy
\xyoption{all}

\title[A bound for Castelnuovo-Mumford regularity by double point divisors]{A bound for Castelnuovo-Mumford regularity \\by double point divisors}

\begin{document}

\author{Sijong Kwak}
\address{Department of Mathematical Sciences, KAIST, Daejeon, Korea}
\email{sjkwak@kaist.ac.kr}

\author{Jinhyung Park}
\address{School of Mathematics, Korea Institute for Advanced Study, Seoul, Korea}
\email{parkjh13@kias.re.kr}

\thanks{S. Kwak was supported by Basic Science Research Program through the National Research Foundation of Korea (NRF) funded by the Ministry of Science and ICT (2015R1A2A2A01004545).}

\subjclass[2010]{Primary 14N05, 13D02, 14N25; Secondary 51N35}

\date{\today}

\keywords{Castelnuovo-Mumford regularity, regularity conjecture, double point divisor, projection, vanishing theorem}

\begin{abstract}
Let $X \subseteq \mathbb{P}^r$ be a non-degenerate smooth projective variety of dimension $n$, codimension $e$, and degree $d$ defined over an algebraically closed field of characteristic zero.
In this paper, we first show that $\text{reg} (\mathcal{O}_X) \leq d-e$, and classify the extremal and the next to extremal cases. Our result reduces the Eisenbud-Goto regularity conjecture for the smooth case to the problem finding a Castelnuovo-type bound for normality. It is worth noting that McCullough-Peeva recently constructed counterexamples to the regularity conjecture by showing that $\text{reg} (\mathcal{O}_X)$ is not even bounded above by any polynomial function of $d$ when $X$ is not smooth.
For a normality bound in the smooth case, we establish that $\text{reg}(X) \leq n(d-2)+1$, which improves previous results obtained by Mumford, Bertram-Ein-Lazarsfeld, and Noma.
Finally, by generalizing Mumford's method on double point divisors, we prove that $\text{reg}(X) \leq d-1+m$, where $m$ is an invariant arising from double point divisors associated to outer general projections. Using double point divisors associated to inner projection, we also obtain a slightly better bound for $\text{reg}(X)$ under suitable assumptions.
\end{abstract}

\maketitle

%\tableofcontents \setcounter{page}{1}

%%%%%%%%%%%%%%%%%%%%%%%%%%%%%%%%%%%%%%%%%%%%%%%%%%%%%%%%%%
\section{Introduction}

Throughout the paper, we work over an algebraically closed field of characteristic zero.
Let $H$ be a very ample divisor on a smooth projective variety $X$. A celebrated vanishing theorem of Serre asserts that there exists an integer $k_0=k_0(H)$ such that if $k \geq k_0$, then
$$
H^i(X, \mathcal{O}_X((k-i)H))=0 ~~\text{ for $i> 0$}.
$$
It is a natural problem to find an upper bound for $k_0(H)$ in terms of geometric invariants of $X$ and $H$. The first main result of this paper provides a sharp effective upper bound for $k_0(H)$ using the \emph{delta genus}
$$
\Delta(X, H):=H^{\dim (X)} + \dim (X) - h^0(X, \mathcal{O}_X(H))
$$
of the polarized pair $(X, H)$.

\begin{theoremalpha}\label{thm-effserre}
Let $X$ be a smooth projective variety, and $H$ be a very ample divisor on $X$. 
If $k \geq \Delta(X, H)+1$, then
$$
H^i(X, \mathcal{O}_X((k-i)H))=0 ~~\text{ for $i> 0$}.
$$
In particular, we have
$$
\chi (X, kH)=h^0(X, kH)~~\text{ for $k \geq \Delta(X, H)$.}
$$
\end{theoremalpha}

If $X \subseteq \P^r$ is a rational normal scroll of dimension $n$ and $H$ is its hyperplane section, then $\Delta(X, H)=0$ and $H^n(X, \mathcal{O}_X(-nH)) \neq 0$. This shows that our bound in Theorem \ref{thm-effserre} is sharp. Furthermore, we classify polarized pairs $(X, H)$ such that $H^i(X, \mathcal{O}_X((\Delta(X, H)-i)H)) \neq 0$ for some $i>0$ in Theorem \ref{thm-oxreg}.

\medskip

We now turn to the regularity conjecture, which is closely related to the effective Serre vanishing problem. The \emph{Castelnuovo-Mumford regularity} of an embedded projective variety $X \subseteq \P^r$ is defined as
$$
\reg(X) := \min \{ k+1 \mid \text{$X \subseteq \P^r$ is  $k$-normal and $\mathcal{O}_X$ is $k$-regular} \}.
$$
We denote by $\reg(\mathcal{O}_X)$ the minimum $k$ such that $\mathcal{O}_X$ is $k$-regular, i.e., $H^i(X, \mathcal{O}_X(k-i))=0$ for $i > 0$. It is a fundamental problem, initiated by Castelnuovo, Mumford, Gruson-Lazarsfeld-Peskine, etc., to find an explicit upper bound for $\reg(X)$ in terms of geometric invariants of $X \subseteq \P^r$. One particular reason to consider this problem is that $\reg(X)$ gives an upper bound for the degree of defining equations of $X$ in $\P^r$ because $\reg(X)$ is the maximal degree of syzygies among minimal generators of the defining ideal $I_{X|\P^r}$. 
We also remark that $\reg(\mathcal{O}_X)$ is the maximal degree of syzygies of the section ring $R(X, \mathcal{O}_X(1)):=\bigoplus_{k \geq 0} H^0(X, \mathcal{O}_X(k))$. An optimal Castelnuovo-type bound for the regularity was suggested by Eisenbud-Goto \cite{EG} (see also \cite[Section 4]{GLP}) as follows:

\begin{conjecturereg}\label{regcon}
Let $X \subseteq \P^r$ be a non-degenerate projective variety of degree $d$ and codimension $e$. Then we have 
$$
\reg(X) \leq d-e+1.
$$
\end{conjecturereg}

This conjecture has been a long-standing and challenging problem, and there has been a considerable amount of interesting partial results. Castelnuovo \cite{C} carried out the fundamental work in this direction for smooth space curves, and Gruson-Lazarsfeld-Peskine \cite{GLP} completely settled the regularity conjecture for integral curves. Pinkham \cite{Pin} and Lazarsfeld \cite{Laz87} verified the conjectured bound for the regularity for smooth surfaces, and Niu \cite{Ni} proved the same result for mildly singular surfaces. For smooth threefolds, Kwak \cite{Kw1}, \cite{Kw2} obtained a weaker bound $\reg(X) \leq d-e+2$ for general case and a sharp bound $\reg(X) \leq d-1$ for codimension $2$ case. This result is extended to threefolds with rational singularities by Niu-Park \cite{NP2}. Slightly weaker bounds for lower dimensional smooth varieties were established by Kwak \cite{Kw3}. It is also an important problem to classify varieties with the maximal or the next to maximal regularity. The classification of projective curves with maximal regularity is given in \cite{GLP}, but almost nothing is known in higher dimensions.

To solve the regularity conjecture, it is sufficient to show that
\begin{enumerate}[\indent $(1)$]
 \item $X \subseteq \P^r$ is $(d-e)$-normal, and
 \item $\reg(\mathcal{O}_X) \leq d-e$.
\end{enumerate}
Until recently, most results on the regularity conjecture have been centered around the problem (1). Notice that if the regularity conjecture holds in dimension $n$, then the inequality (2) in dimension $n+1$ follows (see Lemma \ref{lem-hyperplane}). Thus the inequality (2) holds for any projective surfaces and for projective threefolds with isolated singularities (see \cite[Proposition 2.3]{NP}). For the curve case, we have a stronger result: $\reg(\mathcal{O}_C) \leq \left\lfloor \frac{d-1}{e} \right\rfloor +1 $ (see Corollary \ref{cor-oxvan}).

It turns out that the problem (2) is evidently nontrivial and important. Recently, McCullough-Peeva \cite{MP} constructed counterexamples to the regularity conjecture. Their counterexamples actually show that $\reg(\mathcal{O}_X)$ is not even bounded above by any polynomial functions of the degree $d$ when $X \subseteq \P^r$ is a highly singular projective variety. It is worth to mention that the regularity conjecture is still open for normal projective varieties.

In this paper, we verify the inequality (2) $\reg(\mathcal{O}_X) \leq d-e$ when $X$ is a smooth projective variety of arbitrary dimension.  We also classify the extremal and the next to extremal cases.

\begin{theoremalpha}\label{thm-oxreg}
Let $X \subseteq \P^r$ be a non-degenerate smooth projective variety of degree $d$ and codimension $e$.  Then we have the following:
\begin{enumerate}[$(a)$]
 \item $\reg(\mathcal{O}_X) \leq d-e$.
 \item $\reg(\mathcal{O}_X) = d-e$ if and only if $X \subseteq \P^r$ is a hypersurface or a linearly normal variety with $d= e+1$ or $e+2$.
 \item $\reg(\mathcal{O}_X)=d-e-1$ if and only if $X \subseteq \P^r$ is an isomorphic projection of a projective variety in $(a)$ at one point, a linearly normal variety with $d=e+3$ and $e \geq 2$, or a complete intersection of type $(2,3)$.
\end{enumerate}
\end{theoremalpha}

Observe that $\Delta(X, \mathcal{O}_X(1)) + 1 \leq d-e$ and the equality holds when $X \subseteq \P^r$ is linearly normal. By Mumford's regularity theorem, if $\mathcal{O}_X$ is $k$-regular, then it is $(k+1)$-regular. Thus Theorem \ref{thm-oxreg} \textcolor{RoyalBlue}{$(a)$} is equivalent to Theorem \ref{thm-effserre}, and can be rephrased as 
$$
\reg(\mathcal{O}_X) \leq \Delta(X, \mathcal{O}_X(1)) + 1.
$$
We emphasize that the delta genus $\Delta(X, \mathcal{O}_X(1))$ is not an extrinsic invariant depending on the given embedding $X \subseteq \P^r$ but an intrinsic invariant depending only on the variety $X$ and the very ample line bundle $\mathcal{O}_X(1)$. 

The complete classification of projective varieties appeared in Theorem \ref{thm-oxreg} \textcolor{RoyalBlue}{$(b)$} and \textcolor{RoyalBlue}{$(c)$} is given in Remark \ref{rem-classification}.
By the classification, it is easy to check that those varieties satisfy $\reg(X) =d-e+1$ or $\reg(X) = d-e$.

\medskip

A classical approach for the study of the projective geometry of an embedded projective variety $X \subseteq \P^r$ is to consider a general outer projection $\pi \colon X \to \overline{X} \subseteq \P^{n+1}$ onto a hypersurface of degree $d$. Then $\pi$ is a birational morphism, and the non-isomorphic locus of $\pi$ gives rise to an effective divisor linearly equivalent to the \emph{double point divisor from outer projection}
$$
D_{out} := -K_X + (d-n-2)H,
$$
where $H$ is a hyperplane section of $X \subseteq \P^r$. By varying projection centers, Mumford \cite{BM} proved that $D_{out}$ is base point free. In \cite{Noma2}, Noma extended Mumford's result to the inner projection case. Let $\pi \colon X \dashrightarrow \overline{X} \subseteq \P^{n+1}$ be a general inner projection onto a hypersurface of degree $d-e+1$. Similarly as in the outer projection case, we define the \emph{double point divisor from inner projection}
$$
D_{inn}:=-K_X + (d-n-e-1)H.
$$
The main result of \cite{Noma2} says that $D_{inn}$ is semiample unless $X \subseteq \P^r$ is a scroll over a curve, the second Veronese surface, or a Roth variety. When $D_{inn}$ is semiample, it follows from Kodaira vanishing theorem that $\reg(\mathcal{O}_X) \leq d-e$. This bound can also be easily checked for the second Veronese surface and Roth varieties. In Section \ref{sec-oxreg}, we show a sharp bound for $\reg(\mathcal{O}_X)$ when $X \subseteq \P^r$ is a scroll over a curve using Castelnuovo's genus bound and Ionescu-Toma's result \cite{IT}. Our proof involves complicated calculations, but it is a natural generalization of the arguments in the curve case.

\medskip

By Theorem \ref{thm-oxreg}, the regularity conjecture for the smooth case is reduced to finding a sharp Castelnuovo-type bound for normality. Note that the first general result on a bound for $\reg(X)$ was obtained by Mumford \cite{BM}:  if $X \subseteq \P^r$ is a  non-degenerate smooth projective variety of dimension $n$, codimension $e$, and degree $d$, then 
$$
\reg(X) \leq (n+1)(d-2)+2.
$$
Bertram-Ein-Lazarsfeld \cite{BEL} and Noma \cite{Noma1} refined Mumford's result as 
$$
\reg(X) \leq e(d-e)+1  \text{ if $e \leq n~~$  and }~~ \reg(X) \leq (n+1)(d-n-1)+1  \text{ if $e \geq n+1$}.
$$ 
We remark that not every smooth projective variety of dimension $n$ can be embedded in $\P^{2n}$.
Moreover, by Barth-Larsen theorem, if $e \leq n-1$, then $X$ is simply connected, and if $e \leq n-2$, then $\Pic(X)$ is generated by the hyperplane section. Thus we may assume in general that $e \geq n+1$, and then, the previous best bound is essentially that
$$
\reg(X) \leq (n+1)d + \alpha_n,
$$
where $\alpha_n$ is a constant depending only on $n$.
Our next aim is to improve the previous results.

\begin{theoremalpha}\label{thm-n(d-2)+1}
Let $X \subseteq \P^r$ be a non-degenerate smooth projective variety of dimension $n$, codimension $e \geq 2$, and degree $d$. Then we have
$$
\reg(X) \leq n(d-2)+1=nd-2n+1.
$$
\end{theoremalpha}

The main ingredients of the proof are the classical methods of Castelnuovo \cite{C} and Mumford \cite{BM}, \cite{M} and a vanishing theorem of de Fernex-Ein \cite{dFE}.

\medskip

Finally, to get a better normality bound, we extend Mumford's method on double point divisors in \cite{BM} as Lemma \ref{mumlemout}, and then, show a Castelnuovo-Mumford regularity bound for smooth varieties using an invariant arising from double point divisors. Let $V_{out}$ be a subspace of $H^0(X, \mathcal{O}_X(D_{out}))$ spanned by geometric sections, which are global sections of $\mathcal{O}_X(D_{out})$ whose zero loci are non-isomorphic loci of general outer projections, and $c_k$ be the codimension of the image of the multiplication map
$V_{out} \otimes H^0(X, \mathcal{O}_X(k)) \rightarrow H^0(X, \mathcal{O}_X(D_{out} + kH))$
for any integer $k \geq 0$. We define
$$
m:= \min \{c_k + k \mid  -2K_X+(d-2n-3+k)H \text{ is nef} \}.
$$

\begin{theoremalpha}\label{thm-norm-out}
Let $X \subseteq \P^r$ be a smooth projective variety of degree $d$.
Then we have
$$\reg(X) \leq d-1+m.$$
\end{theoremalpha}

Concerning a bound for $m$, we show that $-2K_X + (d-2n-3+k)H$ is nef for $k \geq d-1$ and $c_k=0$ for $k \geq n(d-3)$ in Propositions \ref{d-1nef} and \ref{n(d-3)zero}. See Remark \ref{rem-expectation} for further discussions.

We also have a similar result for the inner projection case.
Let $V_{inn}$ be a subspace of $H^0(X, \mathcal{O}_X(D_{inn}))$ spanned by geometric sections, and $c_k'$ be the codimension of the image of the map
$V_{inn} \otimes H^0(X, \mathcal{O}_X(k)) \rightarrow H^0(X, \mathcal{O}_X(D_{inn} + kH))$
for any integer $k \geq 0$. We define
$$
m':= \min \{c_k' + k \mid  -2K_X+(d-2n-e-2+k)H \text{ is nef} \}.
$$

\begin{theoremalpha}\label{thm-norm-inn}
Let $X \subseteq \P^r$ be a non-degenerate smooth projective variety of dimension $n \geq 2$, codimension $e \geq 2$ and degree $d \geq e+3$ such that it is neither a scroll over a smooth projective curve, the second Veronese surface, a Roth variety, nor a complete intersection of type $(2, 3)$. Suppose that $V_{inn}$ is base point free and the multiplication map
$$
H^0(\P^r, \mathcal{O}_{\P^r}(i)) \otimes H^0(X, \mathcal{O}_X(K_X+(n-1)H)) \longrightarrow H^0(X, \mathcal{O}_X(K_X + (n-1+i)H))
$$ 
is surjective for $i=1,2$. Then we have
$$\reg(X) \leq d-e-1+m'.$$
\end{theoremalpha}

We remark that if the conditions of Theorem \ref{thm-norm-inn} are fulfilled and $m' \leq 2$, then the conjectured bound $\reg(X) \leq d-e+1$ holds.

\medskip

The rest of the paper is organized as follows: We start in Section \ref{sec-prelim} by recalling relevant basic facts including the positivity properties of double point divisors. Section \ref{sec-oxreg} is devoted to the proofs of Theorems \ref{thm-effserre} and \ref{thm-oxreg}.  In Section \ref{sec-ndbound}, we prove Theorem \ref{thm-n(d-2)+1}. Finally, in Section \ref{sec-normality}, we generalize Mumford's method, and show Theorems \ref{thm-norm-out} and \ref{thm-norm-inn}.

%%%%%%%%%%%%%%%%%%%%%%%%%%%%%%%%%%%%%%%%%%%%%%%%%%%%%%%%%%
\section{Preliminaries}\label{sec-prelim}

In this section, we collect basic definitions and facts that are useful throughout the paper.

%%%%%%%%%%%%%%%%%%%%%%%%%%%%%%%%%%%%%%%%%%%%%%%%%%%%%%%%%%
\subsection{Castelnuovo-Mumford regularity}
Let $X$ be a projective variety, and $L$ be a very ample line bundle on $X$. 
A coherent sheaf $\mathcal{F}$ on $X$ is said to be \emph{$k$-regular} with respect to $L$ in the sense of Castelnuovo-Mumford if $H^i(X, \mathcal{F}\otimes L^{\otimes k-i})=0$ for all $i >0$. By Mumford's regularity theorem (\cite[Theorem 1.8.5]{positivity}), if $\mathcal{F}$ is $k$-regular, then $\mathcal{F}$ is $(k+1)$-regular. We say that an embedded projective variety $X \subseteq \P^r$ is \emph{$k$-regular} if the ideal sheaf $\mathcal{I}_{X|\P^r}$ is $k$-regular with respect to $\mathcal{O}_{\P^r}(1)$. Note that $X \subseteq \P^r$ is $k$-regular if and only if
\begin{enumerate}[\indent $(1)$]
\item $X \subseteq \P^r$ is $(k-1)$-normal, i.e., the natural restriction map 
$$H^0(\P^{r}, \mathcal{O}_{\P^{r}}(k-1)) \rightarrow H^0 (X, \mathcal{O}_X(k-1))$$ 
is surjective, and
\item $\mathcal{O}_X$ is $(k-1)$-regular with respect to $\mathcal{O}_X(1)$.
\end{enumerate}
The \emph{Castelnuovo-Mumford regularity} of $X$, denoted by $\reg(X)$, is the least integer $k$ such that $X \subseteq \P^r$ is $k$-regular. We also denote by $\reg(\mathcal{O}_X)$ the least integer $k$ such that $\mathcal{O}_X$ is $k$-regular. We refer to \cite[Section 1.8]{positivity} for more details.

\begin{lemma}[{cf. \cite[Lemma 2.2]{NP}}]\label{lem-hyperplane}
Let $X \subseteq \P^r$ be a non-degenerate projective variety of dimension $n \geq 2$, and $Y \subseteq \P^{r-1}$ be a general hyperplane section. Fix integers $k_0$ and $i \geq 1$.
Suppose that
$$
\left\{
\begin{array}{ll}
\text{ $Y \subseteq \P^{r-1}$ is $k$-normal  and $H^1(Y, \mathcal{O}_Y(k))=0$} & \text{if $i=1$ }\\
H^{i-1}(Y, \mathcal{O}_Y(k))=H^i(Y, \mathcal{O}_Y(k))=0 & \text{if $i \geq 2$}
\end{array}
\right.
$$
for any integer $k \geq k_0$. Then $H^i(X, \mathcal{O}_X(k))=0$ for $k \geq k_0-1$.
In particular, if $\reg(Y) \leq k$, then $\reg(\mathcal{O}_X) \leq k-1$.
\end{lemma}

\begin{proof}
From the exact sequence
$$
0 \longrightarrow \mathcal{O}_X(-1) \longrightarrow \mathcal{O}_X \longrightarrow \mathcal{O}_Y \longrightarrow 0,
$$
we obtain an exact sequence
\begingroup\makeatletter\def\f@size{10}\check@mathfonts
$$
\cdots \to H^{i-1}(X, \mathcal{O}_X(k)) \to H^{i-1}(Y, \mathcal{O}_Y(k)) \to H^{i}(X, \mathcal{O}_X(k-1)) \to H^{i}(X, \mathcal{O}_X(k))  \to H^{i}(Y, \mathcal{O}_Y(k)) \to \cdots
$$
\endgroup
If $Y \subseteq \P^{r-1}$ is $k$-normal, then $H^0(X, \mathcal{O}_X(k)) \to H^0(Y, \mathcal{O}_Y(k))$ is surjective. By the assumption in the lemma, we have
$$
H^i(X, \mathcal{O}_X(k-1))=H^i(X, \mathcal{O}_X(k))~~\text{ for $k \geq k_0$}.
$$
By Serre vanishing theorem, $H^i(X, \mathcal{O}_X(k))=0$ for $k \gg 0$. Hence the assertion follows.
\end{proof}

By \cite{GLP} and \cite{Laz87}, the regularity conjecture holds for integral curves and smooth surfaces. Lemma \ref{lem-hyperplane} then implies that if $X \subseteq \P^r$ is a non-degenerate projective variety of dimension $n \geq 2$, codimension $e$, and degree $d$, then we have
$$
H^{n-1}(X, \mathcal{O}_X(k-n+1))=H^n(X, \mathcal{O}_X(k-n))=0~~\text{ for $k \geq d-e$.}
$$
If furthermore $n \geq 3$ and the singular locus of $X$ has codimension $\geq 3$ in $X$, then 
$$
H^{n-2}(X, \mathcal{O}_X(k-n+2))=0~~\text{ for $k \geq d-e$.}
$$

%%%%%%%%%%%%%%%%%%%%%%%%%%%%%%%%%%%%%%%%%%%%%%%%%%%%%%%%%%
\subsection{Castelnuovo's genus bound}
The following classical result plays an important role in proving Theorem \ref{thm-oxreg}.

\begin{theorem}[Castelnuovo]\label{thm-castelgenus}
Let $C \subseteq \P^r$ be a non-degenerate projective curve of degree $d$ and arithmetic genus $g$, and set $m := \left\lfloor \frac{d-1}{r-1} \right\rfloor$ and $\epsilon := (d-1)- m(r-1)$. Then we have
$$
g \leq {m \choose 2}(r-1) + m \epsilon=\frac{(d-1-\epsilon)(d+\epsilon-r)}{2(r-1)}.
$$
If the equality holds, then $C \subseteq \P^r$ is projectively normal.
\end{theorem}

\begin{corollary}\label{cor-oxvan}
Let $C \subseteq \P^r$ be a non-degenerate smooth projective curve of degree $d$ and codimension $e$. Then $\reg(\mathcal{O}_C) \leq \left\lfloor \frac{d-1}{e} \right\rfloor + 1$.
\end{corollary}

\begin{proof}
Let $g$ be the genus of $C$, and set $m := \left\lfloor \frac{d-1}{e} \right\rfloor$ and $\epsilon := (d-1)- me$. By applying Theorem \ref{thm-castelgenus}, we see that
$2g-2 \leq m(me-e+2\epsilon) -2 < m(me+\epsilon +1) = md.$
This implies that $H^1(C, \mathcal{O}_C(m)) = 0$. Thus $\reg(\mathcal{O}_C) \leq m+1$.
\end{proof}

%%%%%%%%%%%%%%%%%%%%%%%%%%%%%%%%%%%%%%%%%%%%%%%%%%%%%%%%%%
\subsection{Roth varieties}\label{subsec-roth}
We recall Ilic's construction of Roth varieties in \cite[Theorem 3.7]{Il}. Let $S:=\P(\mathcal{O}_{\P^1}^{\oplus 2} \oplus \mathcal{O}_{\P^1}(a_1) \oplus \cdots \oplus \mathcal{O}_{\P^1}(a_{n-1}))$ be a rational scroll for all $a_i \geq 1$ with the projection $\pi_1 \colon S \to \P^1$.
Consider the birational morphism $\pi_2 \colon S \rightarrow \overline{S} \subseteq \P^N$ given by the complete linear system $|\mathcal{O}_{S}(1)|$ of the tautological line bundle, where $N=a_1 + \cdots + a_{n-1}+n$.
% Then the singular locus of $\overline{S}$ is a line $L \subseteq \P^r$.
For any integer $b \geq 1$, consider a smooth variety $\widetilde{X} \in |\mathcal{O}_S(b) \otimes \pi_1^*\mathcal{O}_{\P^1}(1)|$ such that $\pi_2|_{\widetilde{X}} \colon \widetilde{X} \xrightarrow{\sim} \pi_2(\widetilde{X})=:X$ is an isomorphism.  Then $X \subseteq \P^N$ is a non-degenerate linearly normal smooth projective variety of dimension $n$. A \emph{Roth variety} $X \subseteq \P^r$ is an isomorphic projection of $X \subseteq \P^N$. Note that 
$\deg(X)=b(N-n)+1=b(a_1 + \cdots + a_{n-1}) +1.$ 
If $b=1$, then $X$ is a rational scroll. For more details, we refer to \cite[Section 3]{Il}.

%%%%%%%%%%%%%%%%%%%%%%%%%%%%%%%%%%%%%%%%%%%%%%%%%%%%%%%%%%
\subsection{Double point divisors from outer projection}\label{subset-dpdout}
Let $X \subseteq \P^r$ be a non-degenerate smooth projective variety of dimension $n$, codimension $e$, and degree $d$, and $H$ be its hyperplane section. Take a general outer projection
$$
\pi=\pi_{\Lambda} \colon X \rightarrow \overline{X} \subseteq \P^{n+1}
$$
centered at an $(e-2)$-dimensional general linear subspace
$\Lambda \subseteq \P^{r}$ with $\Lambda \cap X = \emptyset$. 
Note that $\deg(\overline{X})=d$.
By the birational double point formula (see e.g., \cite[Lemma 10.2.8]{positivity}), the non-isomorphic locus of $\pi$ defines
an effective divisor $D_{out}(\Lambda)$ linearly equivalent to the \emph{double point divisor from outer projection}
$$D_{out}:= -K_X + (d-n-2)H.$$

\begin{lemma}\label{lem-dpdout=0}
$-K_X+(d-n-2)H=0$ if and only if $X \subseteq \P^r$ is a hypersurface.
\end{lemma}

\begin{proof}
The `if' direction is trivial. For the converse, we assume that $K_X=(d-n-2)H$. Let $C \subseteq \P^{e+1}$ be a general curve section of genus $g$. Then $K_X=(d-3)H$ so that $2g-2=d(d-3)$.
 However, if $e \geq 2$, then Theorem \ref{thm-castelgenus} implies that $2g-2<d(d-3)$. Thus $e=1$.
\end{proof}

The global section $s(\Lambda) \in H^0(X, \mathcal{O}_X(D_{out}))$ with div$(s(\Lambda))=D_{out}(\Lambda)$ is called a \emph{geometric section}, and the the effective divisor $D_{out}(\Lambda)$ is called a \emph{geometric divisor}. Let $V_{out}$ be the linear subspace in $H^0(X, \mathcal{O}_X(D_{out}))$ spanned by geometric sections.  By varying centers of projections, Mumford proved the following.

\begin{proposition}[{\cite[Technical appendix 4]{BM}}]\label{prop-dpdoutbpf}
$V_{out}$ is base point free. In particular, $D_{out}$ is base point free.
\end{proposition}

%%%%%%%%%%%%%%%%%%%%%%%%%%%%%%%%%%%%%%%%%%%%%%%%%%%%%%%%%%
\subsection{Double point divisors from inner projection}\label{subsec-dpdinn}
Let $X \subseteq \P^r$ be a non-degenerate smooth projective variety of dimension $n \geq 2$, codimension $e \geq 2$, and degree $d$, and $H$ be a general hyperplane section. Take a general inner projection 
$$
\pi_{\Lambda}\colon X\dashrightarrow \overline{X} \subseteq \P^{n+1}
$$
centered at $e-1$ general points $x_1, \ldots, x_{e-1}$ on $X$. 
Note that $\deg(\overline{X})=d-e+1$.
Let $\sigma \colon \widetilde{X} \to X$ be the blow-up at $x_1, \ldots, x_{e-1}$. We have a birational morphism $\widetilde{\pi} \colon \widetilde{X} \to \overline{X}$. We obtain the following commutative diagram
\[\xymatrix{
\widetilde{X} \ar^{\sigma}[r] \ar_{\widetilde{\pi}}[rd] & X \ar@{.>}[d]^-{\pi_{\Lambda}} \ar@{^{(}->}[r] & \P^{r} \ar@{.>}[d]^-{\pi_{\Lambda}}\\
& \overline{X} \ar@{^{(}->}[r] & \P^{n+1}.
}\]
Let $\Lambda= \langle x_1, \ldots, x_{e-1} \rangle$ be the linear span by $x_1, \ldots, x_{e-1}$ so that $\Lambda$ is an $(e-2)$-dimensional linear subspace of $\P^r$. 
%By the general position lemma (see e.g., \cite[Lemma 1.2]{Noma2}), we have $\Lambda \cap X = \{x_1, \ldots, x_{e-1}\}$. 

In \cite{Noma2}, Noma proves that the birational morphism $\widetilde{\pi} \colon \widetilde{X} \to X$ contracts some divisors if and only if $X \subseteq \P^r$ is a scroll over a smooth projective curve or the second Veronese surface $v_2(\P^2) \subseteq \P^5$ (see \cite[Theorem 3]{Noma2}). Suppose now that $X \subseteq \P^r$ is neither a scroll over a smooth projective curve nor the second Veronese surface. By the birational double point formula (\cite[Lemma 10.2.8]{positivity}), the non-isomorphic locus of $\widetilde{\pi}$ defines an effective divisor $\widetilde{D}(\Lambda)$ linearly equivalent to $-K_{\widetilde{X}} + \widetilde{\pi}^*K_{\overline{X}}$. We define an effective divisor $D_{inn}(\Lambda):=\sigma_* \widetilde{D}(\Lambda)$ on $X$. Then $D_{inn}(\Lambda)$ is linearly equivalent to the \emph{double point divisor from inner projection}
$$
D_{inn}:=-K_X + (d-n-e-1)H.
$$

\begin{lemma}
$-K_X + (d-n-e-1)H=0$ if and only if $X \subseteq \P^r$ is a linearly normal del Pezzo manifold.
\end{lemma}

\begin{proof}
The `if' direction is trivial. For the converse, we assume that $K_X=(d-n-e-1)H$.
Let $C \subseteq \P^{e+1}$ be a general curve section of genus $g$. Then $K_C = (d-e-2)H$ so that $2g-2=d(d-e-2)$. However, if $d \geq e+3$, then Theorem \ref{thm-castelgenus} implies that $2g-2 < d(d-e-2)$. Thus $d \leq e+2$. By considering the classification of varieties of minimal and almost minimal degree (see Remark \ref{rem-classification}), we can conclude that $X \subseteq \P^r$ is a linearly normal del Pezzo manifold.
\end{proof}

The global section $s(\Lambda) \in H^0(X, \mathcal{O}_X(D_{inn}))$ with div$(s(\Lambda))=D_{inn}(\Lambda)$ is called a \emph{geometric section}, and the the effective divisor $D_{inn}(\Lambda)$ is called a \emph{geometric divisor}. Let $V_{inn}$ be the linear subspace in $H^0(X, \mathcal{O}_X(D_{inn}'))$ spanned by geometric sections.
By varying the centers of projections, Noma shows that the base locus $\Bs(|V_{inn}|)$ of $V_{inn}$ lies in the set of non-birational centers of simple inner projections (see \cite[Theorem 1]{Noma2}). Thus we have
$$
\Bs(|V_{inn}|)\subseteq \mathcal{C}(X) :=\{u \in X\mid \ell (X \cap \langle u,x\rangle) \geq 3 \text{ for a general point } x \in X\}.
$$
By \cite[Corollary 6.2]{Noma2}, we know that if $\dim \mathcal{C}(X) \geq 1$, then $X$ is a Roth variety.

\begin{theorem}[{\cite[Theorem 4]{Noma2}}]\label{thm-dinnsemiample}
Suppose that $X \subseteq \P^r$ is neither a scroll over a smooth projective curve, the second Veronese surface, nor a Roth variety. Then $D_{inn}$ is semiample.
\end{theorem}

%%%%%%%%%%%%%%%%%%%%%%%%%%%%%%%%%%%%%%%%%%%%%%%%%%%%%%%%%%
\section{A sharp Castelnuovo-Mumford regularity bound for the structure sheaf}\label{sec-oxreg}

This section is devoted to the proofs of Theorems \ref{thm-effserre} and \ref{thm-oxreg}. 
Recall that Theorem \ref{thm-effserre} is equivalent to Theorem \ref{thm-oxreg} \textcolor{RoyalBlue}{$(a)$}. Thus we only focus on Theorem \ref{thm-oxreg}.
We start by recalling Noma's result.

\begin{proposition}[{\cite[Corollary 5]{Noma2}}]\label{prop-oxregnoma}
Let $X \subseteq \P^r$ be a non-degenerate smooth projective variety of dimension $n \geq 2$, codimension $e \geq 2$, and degree $d$. Suppose that $X \subseteq \P^r$ is neither a scroll over a smooth projective curve, the second Veronese surface, nor a Roth variety. Then we have
$\reg(\mathcal{O}_X) \leq d-e$.
\end{proposition}

\begin{proof}
We include the proof for reader's convenience. By Theorem \ref{thm-dinnsemiample}, we know that $D_{inn}=-K_X + (d-n-e-1)H$ is semiample so that $-K_X + (d-n-e-1)H + (n+1-i)H$ is ample for $1 \leq i \leq n$. It follows from Kodaira vanishing theorem that
$$
H^i(X, \mathcal{O}_X(d-e-i))=H^i(X, \mathcal{O}_X(K_X + (-K_X + (d-n-e-1)H + (n+1-i)H)))=0
$$
for $1 \leq i \leq n$. Thus the assertion holds.
\end{proof}

If $X=v_2(\P^2) \subseteq \P^5$ is the second Veronese surface, then $\reg(\mathcal{O}_X)=1=\deg(X)-\codim(X)$.
When $X \subseteq \P^r$ is a Roth variety, we can also compute $\reg(\mathcal{O}_X)$.

\begin{proposition}\label{rothoxreg}
Let $X \subseteq \P^r$ be a non-degenerate Roth variety of dimension $n$, codimension $e \geq 2$, and degree $d = b(N-n)+1$ for some integer $b \geq 2$, where it is an isomorphic projection of the linearly normal Roth variety $X \subseteq \P^N$. Then we have $\reg(\mathcal{O}_X)=b$. Furthermore, $\reg(\mathcal{O}_X) \leq d-e-1$ and the equality holds if and only if $b=2, n=3, e=2, d=5$.
\end{proposition}

\begin{proof}
The first assertion is shown in \cite[Corollary 5]{Noma2}, but we include the proof for reader's convenience. We use the notations in Subsection \ref{subsec-roth}. Note that $H^i(X, \mathcal{O}_X(k))=0$ for $0 < i <n$ and $k \in \Z$ (see \cite[Theorem 3.14]{Il}). 
Let $H$ be a general hyperplane section of $X \subseteq \P^r$, and $F$ be the restriction of a fiber of $\pi \colon S \to \P^1$ to $X$. Then $K_X = (b-n-1)H + (N-n-1)F$. By Serre duality, we have
\[
\begin{array}{l}
H^n(X, \mathcal{O}_X(b-1-n))=H^0(X, \mathcal{O}_X(K_X-(b-1-n)H))^*=H^0(X, \mathcal{O}_X((N-n-1)F))^*,\\
H^n(X, \mathcal{O}_X(b-n))=H^0(X, \mathcal{O}_X(K_X - (b-n)H))^*=H^0(X, \mathcal{O}_X(-H + (N-n-1)F))^*.
\end{array}
\]
Now, $(N-n-1)F$ is an effective divisor, and $-H + (N-n-1)F$ is not a pseudoeffective divisor.
Thus we obtain $H^0(X, \mathcal{O}_X((N-n-1)F)) \neq 0$ and $H^0(\mathcal{O}_X(-H + (N-n-1)F))=0$. This proves that $\reg(\mathcal{O}_X)=b$.
For the second assertion, suppose that $b \geq d-e-1$. Then we have
$$
b \geq d-e-1 \geq b(N-n)-e \geq b(N-n)-(N-n),
$$
so we obtain $1 \geq (b-1)(N-n-1)$. Thus $b=N-n=2$ so that $d=5, e=2$. Now, since $a_1 + \cdots + a_{n-1}=2$, it follows that $n=3$. This completes the proof.
\end{proof}

In view of Proposition \ref{prop-oxregnoma}, it only remains to consider the scroll case. To motivate our approach, we first give the proof of Theorem \ref{thm-oxreg} for the curve case.

\begin{proposition}\label{curveoxreg}
Let $C \subseteq \P^r$ be a non-degenerate smooth projective curve of degree $d$, codimension $e$, and genus $g$ Then we have the following:
\begin{enumerate}[$(1)$]
 \item $\reg(\mathcal{O}_C) \leq d-e$.
 \item $\reg(\mathcal{O}_C) = d-e$ if and only if  $C \subseteq \P^r$ is a plane curve, a rational normal curve $(d=e+1)$, or an elliptic normal curve $(d=e+2)$.
 \item $\reg(\mathcal{O}_C) = d-e-1$ if and only if $C \subseteq \P^r$ is an isomorphic projection of a projective variety in $(2)$ at one point, a linearly normal curve of genus $g=2$ and degree $d=e+3$, or a complete intersection of type $(2,3)$.
\end{enumerate}
\end{proposition}

\begin{proof}
Since the assertion is obvious for $e=1$, we assume that $e \geq 2$. Suppose that $d \leq e+3$. By applying Theorem \ref{thm-castelgenus}, we see that $g \leq 2$.  Note that $\reg(\mathcal{O}_C)=1$ if $g=0$ and $\reg(\mathcal{O}_C)=2$ if $g=1,2$. It is then easy to check the proposition. Now, suppose that $d \geq e+4$. We only have to show that $H^1(C, \mathcal{O}_C(d-e-3))=0$ except when $C \subseteq \P^r$ is a complete intersection of type $(2,3)$. It follows from Theorem \ref{thm-castelgenus} and direct calculations that 
$$
(d-e-3)d > 2g-2
$$
except when $(d,e)=(6,2)$. If $(d,e) \neq (6,2)$, then we obtain $H^1(C, \mathcal{O}_C(d-e-3))=0$. If $(d,e)=(6,2)$ and $2g-2 \geq (d-e-3)d = 6$, then $g \geq 4$. We can easily check that the space curve of degree $6$ and genus $g \geq 4$ is a complete intersection of type $(2,3)$. 
\end{proof}

Recall that a \emph{scroll} $X \subseteq \P^r$ over a smooth projective curve $C$ is an isomorphic projection of $X=\P(E) \subseteq \P^N$, where $E$ is a very ample vector bundle on $C$ and the embedding $\P(E) \subseteq \P^N$ is given by the complete linear system $|\mathcal{O}_{\P(E)}(1)|$. %For convenience, we assume from now on that any scroll has dimension $\geq 2$.

Let $E$ be a vector bundle on a smooth projective curve $C$ of genus $g$. We define
$$
\mu^-(E):=\min \left\{ \mu(Q)=\frac{\deg Q}{\rank Q}  \suchthat Q \text{ is a quotient bundle of }E \right\}.
$$
If $E$ is semistable, then $\mu^-(E)=\frac{\deg E}{\rank E}$. If $E$ is not semistable, then there is a semistable quotient bundle $Q$ of $E$ with $\mu^-(E)=\mu(Q)=\frac{\deg Q}{\rank Q}$.

\begin{lemma}[{\cite[Lemmas 1.12 and 2.5]{B}}]\label{mu-lemma}
We have the following:
\begin{enumerate}[$(1)$]
 \item $\mu^-(S^k (E))=k \mu^-(E)$ for any integer $k > 0$.
 \item If $\mu^-(E) > 2g-2$, then $H^1(C, E)=0$.
\end{enumerate}
\end{lemma}

\begin{lemma}\label{lem-scrolloxreg}
Let $X \subseteq \P^r$ be a non-degenerate scroll of degree $d$ and codimension $e$ over a smooth projective curve of genus $g$. Suppose that $n=\dim (X) \geq 2$.  Then we have the following:
\begin{enumerate}[$(1)$]
 \item If $g = 0$, then $\reg(\mathcal{O}_X)=1$.
 \item If $g = 1$, then $\reg(\mathcal{O}_X)=2$.
 \item If $g \geq 2$, then $\reg(\mathcal{O}_X) \leq d-e-2$.
\end{enumerate}
\end{lemma}

\begin{proof}
Let $E$ be a very ample vector bundle on a smooth projective curve $C$ of genus $g$ such that $X = \P(E) \subseteq \P^r$ is a scroll over $C$, and $F$ be a general fiber of the natural projection $\pi \colon \P(E) \to C$. Note that $H^i(X, \mathcal{O}_X(k))=0$ for $1 < i < n$ and $k \in \Z$. We also have $H^n(X, \mathcal{O}_X(k))=0$ for $k > -n$ and $H^n(X, \mathcal{O}_X(k)) \neq 0$ for $k \leq -n$. Thus we only have to consider the vanishing for $H^1(X, \mathcal{O}_X(k))$.

If $g=0$, then $H^1(X, \mathcal{O}_X)=0$, which implies (1). If $g=1$, then $\mu^-(E)>0 = 2g-2$  due to the very ampleness of $E$. By Lemma \ref{mu-lemma}, $H^1(X, \mathcal{O}_X(1))=H^1(C, E)=0$, which implies (2).

It only remains to consider the case that $g \geq 2$. It suffices to show that  
$$
H^1(X, \mathcal{O}_X(d-e-3))=H^1(C, S^{d-e-3}(E))=0.
$$ 
By Lemma \ref{mu-lemma}, it is enough to show that
\begin{equation}\label{mueq}
(d-e-3)\mu^-(E)=\mu^-(S^{d-e-3}(E))>2g-2.
\end{equation}

We consider the embedding $C \subseteq \P^N$ given by $|\det E|$. The main theorem of \cite{IT} asserts that 
$$
N+1=h^0(C, \det E) \geq h^0(C, E) + n -2.
$$
Since $h^0(C, E) \geq r+1$, it follows that $N \geq r+n-2 \geq n+1$. Consider an inner projection of $C \subseteq \P^N$ centered at $N-n-1$ general points to a projective curve $\overline{C} \subseteq \P^{n+1}$. Note that $n+1 \geq 3$. Then we have
$$
\overline{d}:=\deg(\overline{C}) =d-N+n+1 \leq d-(r+n-2)+n+1 = d-r+3 \leq d-e+1.
$$
Since the inner projection map $C \to \overline{C}$ is birational, the genus $g$ of $C$ is less than or equal to the arithmetic genus of $\overline{C}$.
By applying Theorem \ref{thm-castelgenus} to $\overline{C} \subseteq \P^{n+1}$, we obtain
\begin{equation}\label{casteln}
\frac{(d-e-\epsilon)(d+\epsilon-e-n)}{n}-2 \geq \frac{(\overline{d}-1-\epsilon)(\overline{d}+\epsilon-n-1)}{n}-2 \geq  2g-2,
\end{equation}
where $\epsilon = \overline{d}-\left\lfloor \frac{\overline{d}-1}{n} \right\rfloor n -1$. Note that $0 \leq \epsilon \leq n-1$.

First, suppose that $E$ is semistable so that $\mu^-(E)=\mu(E)=\frac{d}{n}$. It is easy to check that
$$
\frac{(d-e-3)d}{n} > \frac{(d-e-\epsilon)(d+\epsilon-e-n)}{n}
$$
according to $\epsilon =0, 1, 2$, and $\epsilon \geq 3$.  By (\ref{casteln}), we obtain
$$
(d-e-2)\mu^-(E) = \frac{(d-e-2)d}{n} > 2g-2.
$$
Thus we verify (\ref{mueq}) in the case that $E$ is semistable.

Now, suppose that $E$ is not semistable so that there is a semistable quotient bundle $Q$ of $E$ with $\mu^-(E)=\mu(Q)$. Note that $Q$ is also very ample.
Let $d':=\deg Q$ and $n' = \rank Q$. Consider the scroll $X'= \P(Q) \subseteq \P^{n'+e'}$, where the embedding is given by $|\mathcal{O}_{\P(Q)}(1)|$. We further divide into two cases: (i) $d-e-3 \geq d'-e'-3$ and (ii) $d-e-3<d'-e'-3$. Suppose that we are in Case (i).
Since $Q$ is semistable, we know that $(d'-e'-3) \mu(Q)> 2g-2$. Thus we have
$$
(d-e-3)\mu^-(E) \geq (d'-e'-3) \mu(Q) > 2g-2,
$$
which verifies (\ref{mueq}).
Suppose that we are in Case (ii). We have $d'>d-e+e'$ and $n>n'$, so we obtain
$$
(d-e-3)\mu^-(E)=(d-e-3)\frac{d'}{n'} > (d-e-3)\frac{(d-e+e')}{n}.
$$
It is straightforward to check
$$
\frac{(d-e-3)(d-e+e')}{n} \geq \frac{(d-e-\epsilon)(d+\epsilon-e-n)}{n}-2
$$
according to $\epsilon = 0,1,2$, and $\epsilon \geq 3$. Now (\ref{casteln}) implies that
$$
(d-e-3)\mu^-(E) > 2g-2,
$$
which verifies (\ref{mueq}).
Therefore, we complete the proof.
\end{proof}

\begin{proposition}\label{prop-scrolloxreg}
Let $X \subseteq \P^r$ be a non-degenerate scroll of degree $d$ and codimension $e$ over a smooth projective curve of genus $g$. Suppose that $n=\dim (X) \geq 2$.  Then we have the following:
\begin{enumerate}[$(1)$]
 \item $\reg(\mathcal{O}_X) \leq d-e$.
 \item $\reg(\mathcal{O}_X) = d-e$ if and only if $X \subseteq \P^r$ is a rational normal scroll $(d=e+1)$.
 \item $\reg(\mathcal{O}_X) = d-e-1$ if and only if $X \subseteq \P^r$ is an isomorphic projection of $(2)$ at one point or an elliptic normal surface scroll with $d=e+3$.
\end{enumerate}
\end{proposition}

\begin{proof}
By Lemma \ref{lem-scrolloxreg}, we only have to consider the case that $g \leq 1$. 
If $g=0$, then $\reg(\mathcal{O}_X)=1 \leq d-e$ and the equality holds if and only if $X \subseteq \P^r$ is a rational normal scroll. Suppose now that $g=1$. Then $d \geq r+1 = n+e+1$, and the equality holds only if $X \subseteq \P^r$ is linearly normal by \cite[Theorem 1.1]{KP}. We have $\reg(\mathcal{O}_X) =2 \leq n \leq d-e-1$
and the equality holds if and only if $n=2, d=e+3$, i.e., $X \subseteq \P^r$ is an elliptic normal scroll with $d=e+3$. This completes the proof.
\end{proof}

\begin{remark}
Let $X = \P(E) \subseteq \P^r$ be a non-degenerate scroll of degree $d$ and codimension $e$ over a smooth projective curve $C$ of genus $g$, and $F$ be a general fiber of the natural projection $\pi \colon X \to C$. Consider the divisor $D := -K_X + (d-n-e-1)H \equiv (d-e-1)H - (d+2g-2)F$.
If $D$ is nef, then Kodaira vanishing theorem implies that $\reg(\mathcal{O}_X) \leq d-n-e-1$, which is stronger than Lemma \ref{lem-scrolloxreg}.
By Miyaoka's criterion (see \cite[Lemma 5.4]{B}), $D$ is nef if and only if $(d-e-1)\mu^-(E) \geq d+2g-2$. The similar argument in the proof of Lemma \ref{lem-scrolloxreg} shows that $D$ is nef when $E$ is semistable. However, we do not know whether $D$ is nef in general.
\end{remark}

%%%%%%%%%%%%%%%%%%%%%%%%%%%%%%%%%%%%%%%%%%%%%%%%%%%%%%%%%%
We are ready to prove Theorem \ref{thm-oxreg}.

\begin{proof}[Proof of Theorem \ref{thm-oxreg}]
Let $X \subseteq \P^r$ be a non-degenerate smooth projective variety of dimension $n$, codimension $e$, and degree $d$, and $H$ be its hyperplane section. If $X \subseteq \P^r$ is a hypersurface, then the assertion is trivial. We already show the assertion for the curve case in Proposition \ref{curveoxreg}. Thus we assume that $e \geq 2$ and $n \geq 2$. 

If $X \subseteq \P^r$ is the second Veronese surface $v_2(\P^2) \subseteq \P^5$, then $d=e+1$ and $\reg(\mathcal{O}_X)=1=d-e$. If $X \subseteq \P^r$ is a Roth variety or a scroll over a smooth projective curve, then the assertion follows from Propositions \ref{rothoxreg} and \ref{prop-scrolloxreg}, respectively. Thus we further assume that $X \subseteq \P^r$ is neither the second Veronese surface, a Roth variety, nor a scroll over a smooth projective curve.

\medskip

\noindent $(a)$  The assertion follows from Proposition \ref{prop-oxregnoma}.

\medskip

\noindent $(b)$ The `if' part is trivial by the classification of varieties with $d \leq e+2$ (see Remark \ref{rem-classification}). For the `only if' part, we assume that $\reg(\mathcal{O}_X)=d-e$, i.e., $\mathcal{O}_X$ fails to be $(d-e-1)$-regular. In this case, we observe that $X \subseteq \P^r$ is linearly normal. To see this, suppose that $X \subseteq \P^r$ is not linearly normal. Then it is an isomorphic projection of $X \subseteq \P^{r+1}$, which has degree $d$ and codimension $e+1$. By (a) for $X \subseteq \P^{r+1}$, we have $\reg(\mathcal{O}_X) \leq d-(e+1)=d-e-1$, so we get a contradiction.

By Theorem \ref{thm-dinnsemiample}, $D_{inn}=-K_X + (d-n-e-1)H$ is semiample. Then the divisor 
$$
-K_X + (d-n-e-1)H + (n-i)H=-K_X + (d-e-1-i)H
$$ 
is ample for $1 \leq i \leq n$. It follows from Kodaira vanishing theorem that
$$
H^i(X, \mathcal{O}_X(d-e-1-i))=0 ~~\text{ for $1 \leq i \leq n-1$}.
$$
Thus we must have
$$
H^n(X, \mathcal{O}_X(d-e-1-n)) \neq 0.
$$
Let $Y \subseteq \P^{r-1}$ be a smooth general hyperplane section of $X \subseteq \P^r$. Consider the following exact sequence
$$
0 \longrightarrow \mathcal{O}_X(d-e-n-1) \longrightarrow \mathcal{O}_X(d-e-n) \longrightarrow \mathcal{O}_Y(d-e-n) \longrightarrow 0.
$$
Since $\mathcal{O}_X$ is $(d-e)$-regular, we have $H^n(X, \mathcal{O}_X(d-e-n))=0$. Thus the map
$$
H^{n-1}(Y, \mathcal{O}_Y(d-e-1-(n-1))) \to H^n (X, \mathcal{O}_X(d-e-n-1))
$$
is surjective, so we obtain $H^{n-1}(Y, \mathcal{O}_Y(d-e-1-(n-1))) \neq 0$.
Consequently, for a general smooth curve section $C \subseteq \P^{e+1}$ of $X \subseteq \P^r$, we can conclude that $H^1(C, \mathcal{O}_C(d-e-2)) \neq 0$. By Proposition \ref{curveoxreg}, $C \subseteq \P^{e+1}$ is either a rational normal curve or an elliptic normal curve. Thus $d \leq e+2$, so we complete the proof for $(b)$.

\medskip

\noindent $(c)$ The `if' part is trivial by the classification result in Remark \ref{rem-classification}.
For the `only' if part, we assume that $\reg(\mathcal{O}_X) =d-e-1$. In particular, $\mathcal{O}_X$ fails to be $(d-e-2)$-regular.
If $X \subseteq \P^r$ is not linearly normal, then it is an isomorphic projection of a projective variety $X \subseteq \P^{r+1}$ of degree $d$ and codimension $e+1$. Since $\reg(\mathcal{O}_X)=d-(e+1)$, it follows that $X \subseteq \P^{r+1}$ is a projective variety in (b). Thus we may assume that $X \subseteq \P^r$ is linearly normal.

By Theorem \ref{thm-dinnsemiample}, $D_{inn}=-K_X + (d-n-e-1)H$ is semiample.
Then the divisor
$$
-K_X + (d-n-e-1)H+(n-1-i)H = -K_X + (d-e-2-i)H
$$ 
is ample for $1 \leq i \leq n-2$.
It follows from Kodaira vanishing theorem that
$$
H^i(X, \mathcal{O}_X(d-e-2-i))=0 ~~\text{ for $1 \leq i \leq n-2$}.
$$
Thus we have
$$
H^n(X, \mathcal{O}_X(d-e-2-n)) \neq 0 ~~\text{ or }~~ H^{n-1}(X, \mathcal{O}_X(d-e-1-n)) \neq 0.
$$
Note that $\reg(\mathcal{O}_X)=d-e-1$ implies that 
$$
H^n(X, \mathcal{O}_X(d-e-1-n)) =0~~\text{ and }~~H^{n-1}(X, \mathcal{O}_X(d-e-n))=0.
$$
Let $Y \subseteq \P^{r-1}$ be a smooth general hyperplane section of $X \subseteq \P^r$.
If $H^n(X, \mathcal{O}_X(d-e-2-n)) \neq 0$, then by considering the following short exact sequence
$$
0 \longrightarrow \mathcal{O}_X(d-e-2-n) \longrightarrow \mathcal{O}_X(d-e-1-n) \longrightarrow \mathcal{O}_Y(d-e-1-n) \longrightarrow 0,
$$
we see that $H^{n-1}(Y, \mathcal{O}_Y(d-e-2-(n-1))) \neq 0$. Similarly, one can also check that if $H^{n-1}(X, \mathcal{O}_X(d-e-1-n)) \neq 0$, then $H^{n-2}(Y, \mathcal{O}_Y(d-e-1-(n-1))) \neq 0$.
Consequently, for a general smooth surface section $S \subseteq \P^{e+2}$ of $X \subseteq \P^r$, we have 
$$
H^2(S, \mathcal{O}_S(d-e-4)) \neq 0~~ \text{ or }~~  H^1(S, \mathcal{O}_S(d-e-3)) \neq 0.
$$
Note that $H^2(S, \mathcal{O}_S(d-e-3)=0$ and $H^1(S, \mathcal{O}_S(d-e-2)=0$.

\smallskip

We now claim that
\begin{equation}\label{d-e-3surf}
H^2(S, \mathcal{O}_S(d-e-4)) \neq 0~~ \text{ and } ~~ H^1(S, \mathcal{O}_S(d-e-3)) = 0.
\end{equation}
To show the claim, it is sufficient to prove that $H^1(S, \mathcal{O}_S(d-e-3)) = 0$.
By abuse of notation, we denote by $H$ a general hyperplane section of $S \subseteq \P^{e+2}$.
Since $X \subseteq \P^r$ is neither a second Veronese surface nor a scroll over a curve, so is $S  \subseteq \P^{e+2}$. If $S  \subseteq \P^{e+2}$ is a Roth variety, then $H^1(S, \mathcal{O}_S(k))=0$ for all $k \in \Z$. Thus, by Theorem \ref{thm-dinnsemiample}, we may assume that $D_{inn} = -K_S + (d-e-3)H$ is semiample.

We now show that $D_{inn}$ is big possibly except one case. Consider the divisors $D:=(d-e-2)H$ and $E:=K_S + H$ on $S$. By \cite[Theorem 1.4]{Io}, if $E$ is not base point free, then $S$ is a second Veronese surface, a quadric hypersurface, or a scroll over a curve. Those cases are already excluded, so we may assume that $E$ is base point free. Let $C \subseteq \P^{e+1}$ be a general smooth curve section of $S \subseteq \P^{e+2}$, and $g$ be the genus of $C$. We have
$$
D^2 - 2D.E = (d-e-2)^2d - 2(d-e-2)(K_S.H + H^2)=(d-e-2)\{(d-e-2)d - 4g +4\}.
$$
The following inequality
\begin{equation}\label{dgeq}
\frac{(d-e-2)d + 4}{4} >g,
\end{equation}
implies that $D^2-2D.E>0$ so that $D-E=D_{inn}$ is big by \cite[Theorem 2.2.15]{positivity}. Recall from Theorem \ref{thm-castelgenus} that
$$
\frac{(d-1-\epsilon)(d+\epsilon-e-1)}{2e} \geq g,
$$
where $\epsilon = d-1-\lfloor \frac{d-1}{e} \rfloor e$ and $0 \leq \epsilon \leq e-1$. 
To show (\ref{dgeq}), it is enough to prove that
$$
\frac{(d-e-2)d + 4}{4} > \frac{(d-1-\epsilon)(d+\epsilon-e-1)}{2e},
$$
which is equivalent to 
\begin{equation}\label{dgeq2}
(e-2)(d-e-2)d+4e > 2(e+1-\epsilon)(\epsilon+1).
\end{equation}
We divide into two cases according to $e=2$ and $e \geq 3$. 
Suppose that $e \geq 3$. If $d \leq e+2$, then $H^1(S, \mathcal{O}_S(k))=0$ for $k \in \Z$. Thus we can assume that $d \geq e+3$. If $\epsilon \leq 2$, then $2(e+1-\epsilon)(\epsilon+1) \leq 6e-6$. Since we have $(e-2)(d-e-2)d > 2(e-3)$, we get (\ref{dgeq2}). If $\epsilon \geq 3$, then $d \geq e+4$. Since we have $(e-2)(d-e-2)d>(e+1-\epsilon)2(\epsilon+1)$, we also get (\ref{dgeq2}).
Suppose now that $e=2$. When $\epsilon =0$, it is straightforward to verify (\ref{dgeq2}).
We have shown that $D_{inn}$ is big except when $(e, \epsilon)=(2,1)$. Thus $D_{inn}$ is now nef and big possibly except one case, and hence, it follows from Kawamata-Viehweg vanishing theorem that $H^1(S, \mathcal{O}_S(d-e-3)) = 0$.

It only remains to consider the case that $(e, \epsilon)=(2,1)$.
Theorem \ref{thm-castelgenus} says that $\frac{(d-2)^2}{4} \geq g$. If (\ref{dgeq}) does not hold, then we must have $\frac{(d-2)^2}{4} = g$. By Theorem \ref{thm-castelgenus} again, $C \subseteq \P^{e+1}$ is projectively normal. This implies that the natural restriction map 
$$
H^0(S, \mathcal{O}_S(d-e-2)) \longrightarrow H^0(C, \mathcal{O}_C(d-e-2))
$$ 
is surjective. Recall that $H^1(S, \mathcal{O}_S(d-e-2))=0$. By considering an exact sequence
$$
0 \longrightarrow \mathcal{O}_S(d-e-3) \longrightarrow  \mathcal{O}_S(d-e-2) \longrightarrow  \mathcal{O}_C(d-e-2) \longrightarrow  0,
$$
we get $H^1(S, \mathcal{O}_S(d-e-3)) = 0$. 
This completes the proof of the claim (\ref{d-e-3surf}).

\smallskip

We continue the proof of Theorem \ref{thm-oxreg} \textcolor{RoyalBlue}{$(c)$}.
Let $C \subseteq \P^{e+1}$ be a general smooth curve section of $S \subseteq \P^{e+2}$, and $g$ be the genus of $C$.
By the claim (\ref{d-e-3surf}),  we see that $H^1(C, \mathcal{O}_C(d-e-3)) \neq 0$ so that $\mathcal{O}_C$ fails to be $(d-e-2)$-regular. By Proposition \ref{curveoxreg}, we have either $d \leq e+3, g \leq 2$ or $e=2, d=6, g=4$. 
Consider the first case ($d \leq e+3, g \leq 2$). We already assume that $X \subseteq \P^r$ is linearly normal and is not a scroll over a curve. If $g \leq 1$, then $X \subseteq \P^r$ is a variety of minimal or almost minimal degree, so $\reg(\mathcal{O}_X)=d-e \neq d-e-1$. Thus $g=2$, and consequently, $d=e+3$.
In the second case ($e=2, d=6, g=4$), we can check that $X \subseteq \P^r$ is a complete intersection of type $(2,3)$. Therefore, we complete the proof of Theorem \ref{thm-oxreg} \textcolor{RoyalBlue}{$(c)$}.
\end{proof}

\begin{remark}\label{rem-classification}
Let $X \subseteq \P^r$ be a smooth projective variety of degree $d$ and codimension $e$. Note that $d \geq e+1$.
\begin{enumerate}[$(1)$]
\item  $d=e+1$: $X \subseteq \P^r$ is called a \emph{variety of minimal degree}. By del Pezzo-Bertini (see e.g., \cite{EH}), it is known that such a variety is either a quadric hypersurface, the second Veronese surface $v_2 (\P^2) \subseteq \P^5$, a rational normal curve, or a rational normal scroll.
\item $d=e+2$: $X \subseteq \P^r$ is called a \emph{variety of almost minimal degree}. Such a variety is either an isomorphic projection from a variety of minimal degree at one point, an elliptic normal curve, or a linearly normal del Pezzo manifold. Fujita completely classified del Pezzo manifolds in  \cite{F1} and \cite{F2}.
\item $d=e+3$: If $X \subseteq \P^r$ is not linearly normal, then it is obtained from an isomorphic projection of a variety of degree $d \leq e+2$. Ionescu classified linearly normal varieties  $X \subseteq \P^r$ of degree $d=e+3$ (see \cite[Theorem 3.12]{Io}) as follows: $X \subseteq \P^r$ is either a quartic hypersurface, an elliptic normal surface scroll, or a linearly normal variety with sectional genus $g=2$ which is not a scroll over a curve. The last case is either a blowing-up of Hirzebruch surfaces $F_n=\P(\mathcal{O}_{\P^1} \oplus \mathcal{O}_{\P^1}(-n))$ with $0 \leq n \leq 2$ at $k \leq 7$ points (see \cite[Proposition 3.1]{Io}) or one of 6 cases in higher dimensions (see \cite[Theorem 3.4]{Io}).
\end{enumerate}
\end{remark}

%%%%%%%%%%%%%%%%%%%%%%%%%%%%%%%%%%%%%%%%%%%%%%%%%%%%%%%%%%
\section{A Castelnuovo-Mumford regularity bound for smooth projective varieties}\label{sec-ndbound}

The aim of this section is to prove Theorem \ref{thm-n(d-2)+1}. 
The following is used in \cite{C} and \cite[Technical appendix 4]{BM} (see also \cite[Section II]{Pin}), but we include the proof for reader's convenience.

\begin{lemma}\label{outidealsurj}
Let $D \in |V|$ be an effective divisor on $X$. Then the natural restriction map
$$
H^0(\P^r, \mathcal{I}_{D|\P^r}(d-1+\ell)) \longrightarrow H^0(X, \mathcal{I}_{D|X}(d-1+\ell))
$$
is surjective for any integer $\ell$.
\end{lemma}

\begin{proof}
We first consider the case that $D=D_{out}(\Lambda)$ is a geometric divisor associated to a general outer projection $\pi=\pi_{\Lambda} \colon X \to \overline{X} \subseteq \P^{n+1}$ centered at a general linear subspace $\Lambda \subseteq \P^r$. Let $\overline{D} := \pi (D)$ be the image of $D$ under the projection $\pi$. Then $\overline{D}$ is a Weil divisor on $\overline{X}$. We have the following diagram:
\[\xymatrix{
D \ar@{^{(}->}[r] \ar^{\pi|_D}[d]  & X \ar@{^{(}->}[r] \ar^-{\pi}[d] & \P^r \ar@{.>}[d]^-{\pi}\\
\overline{D}  \ar@{^{(}->}[r] & \overline{X}  \ar@{^{(}->}[r] & \P^{n+1}.
}\]
Note that $\mathcal{I}_{\overline{D}|\overline{X}} \simeq \pi_* \mathcal{I}_{D|X} = \pi_* \mathcal{O}_X(-D)$ is the conductor ideal sheaf of the finite birational morphism $\pi$ (see \cite[Technical appendix 4]{BM}).
Consider the following short exact sequence
$$
0 \longrightarrow \mathcal{I}_{\overline{X}|\P^{n+1}} \longrightarrow \mathcal{I}_{\overline{D}|\P^{n+1}} \longrightarrow \mathcal{I}_{\overline{D}|\overline{X}} \longrightarrow 0.
$$
We have the following commutative diagram:
\[\xymatrix{
H^0(\P^{n+1}, \mathcal{I}_{\overline{D}|\P^{n+1}}(d-1+\ell)) \ar@{^{(}->}[d] \ar@{->>}[r] &  H^0(\overline{X}, \mathcal{I}_{\overline{D}|\overline{X}}(d-1+\ell)) \ar@{=}[d] \\
H^0(\P^r, \mathcal{I}_{D|\P^r}(d-1+\ell)) \ar[r]  & H^0(X, \mathcal{I}_{D|X}(d-1+\ell)).
}\]
Since $\mathcal{I}_{\overline{X}|\P^{n+1}} \simeq \mathcal{O}_{\P^{n+1}}(-d)$ and $H^1(\P^{n+1}, \mathcal{O}_{\P^{n+1}}(\ell-1))=0$, it follows that the upper horizontal map is surjective. Thus the lower horizontal map is also surjective. 

Now, let $s(\Lambda) \in V$ be a geometric section associated to a general outer projection $\pi=\pi_{\Lambda}$. We may regard $H^0(X, \mathcal{I}_{D_{out}(\Lambda)|X}(d-1+\ell))$ as a linear subspace $s(\Lambda) \cdot H^0 (X, \mathcal{O}_X(K_X + (n+1+\ell)H))$ of $H^0(X, \mathcal{O}_X(d-1+\ell))$. We have shown in the previous paragraph that every global section in $s(\Lambda) \cdot H^0 (X, \mathcal{O}_X(K_X + (n+1+\ell)H))$ is lifted to a global section in $H^0(\P^r, \mathcal{O}_{\P^r}(d-1+\ell))$. Since every global section $s \in V$ is a linear sum of geometric sections, it follows that every global section in $s \cdot H^0 (X, \mathcal{O}_X(K_X + (n+1+\ell)H))$ is lifted to a global section in $H^0(\P^r, \mathcal{O}_{\P^r}(d-1+\ell))$. This implies the lemma.
\end{proof}

\begin{corollary}\label{regdpd}
Let $	X \subseteq \P^r$ be a smooth projective variety of degree $d$, and $D \in |V_{out}|$ be an effective divisor on $X$. Then $X \subseteq \P^r$ is $(d-1+\ell)$-normal for some integer $\ell$ if and only if the map 
$$
H^1(\P^r, \mathcal{I}_{D|\P^r}(d-1+\ell)) \longrightarrow H^1(X, \mathcal{I}_{D|X}(d-1+\ell))
$$
is injective. In particular, if $D \subseteq \P^r$ is $(d-1+\ell)$-normal, then so is $X \subseteq \P^r$.
\end{corollary}

\begin{proof}
We consider an exact sequence
$$
0 \longrightarrow \mathcal{I}_{X|\P^r}  \longrightarrow \mathcal{I}_{D|\P^r}  \longrightarrow \mathcal{I}_{D|X} \longrightarrow 0.
$$
By Lemma \ref{outidealsurj}, the above exact sequence induces the following exact sequence
$$
0 \longrightarrow H^1(\P^r, \mathcal{I}_{X|\P^r}(d-1+\ell))  \longrightarrow H^1(\P^r, \mathcal{I}_{D|\P^r}(d-1+\ell)) \longrightarrow H^1(X, \mathcal{I}_{D|X}(d-1+\ell)) \longrightarrow \cdots.
$$
Then the assertion immediately follows.
\end{proof}

In \cite{M}, Mumford proved that any smooth projective variety $X \subseteq \P^r$ of degree $d$ is scheme-theoretically cut out by hypersurfaces of degree $d$ in $\P^r$. We need a generalization of this result to some singular varieties.

\begin{lemma}\label{lem-cutoutd}
Let $X \subseteq \P^r$ be a projective variety of degree $d$. Suppose that $X$ has only finitely many singular points and the tangent cone of each singular point is a finite union of distinct linear subspaces of $\P^r$. Then $X \subseteq \P^r$ is scheme-theoretically cut out by hypersurfaces of degree $d$ in $\P^r$.
\end{lemma}

\begin{proof}
Let $n:=\dim (X)$. As in \cite{M}, we consider a linear subspace $\Lambda$ of dimension $r-n-2$ in $\P^r$ disjoint from $X$ and the join $H_{\Lambda}$ of $X$ and $\Lambda$ (the locus of lines joining $X$ and $\Lambda$). Then $H_{\Lambda}$ is a hypersurface of degree $\leq d$ in $\P^r$. It is easy to see that
$$
X = \bigcap_{\Lambda \cap X = \emptyset} H_{\Lambda}
$$
as sets. By the conditions on singularities of $X$, we have
$$
TC_{p}(X) = \bigcap_{\Lambda \cap X = \emptyset} TC_{p}(H_{\Lambda})~~\text{ for any point $p \in X$},
$$
where $TC$ stands for the tangent cone. Thus the assertion follows.
\end{proof}

\begin{remark}\label{rem-dfE}
Let $X \subseteq \P^r$ be a non-degenerate smooth projective variety of dimension $n$, codimension $e \geq n+1$, and degree $d$. We take a general outer projection
$$
\pi \colon X \to X' \subseteq \P^{2n}.
$$
We then claim that $X' \subseteq \P^{2n}$ is $k$-normal for $k \geq n(d-2)$. This claim plays a crucial role in proving Theorem \ref{thm-n(d-2)+1}.
Notice that $X'$ has only finitely many double points and the tangent cone of each singular point is the union of two $n$-dimensional linear subspaces of $\P^{2n}$ meeting at one point (see \cite[Example 5.7]{dFE}). By Lemma \ref{lem-cutoutd}, $X' \subseteq \P^{2n}$ is scheme-theoretically cut out by hypersurfaces of degree $d$ in $\P^{2n}$.
Since the pair $(\P^{2n}, nX')$ is log canonical by \cite[Example 5.7]{dFE},  it follows from \cite[Corollary 5.1]{dFE} that
$$
H^1(\P^{2n}, \mathcal{I}_{X'|\P^{2n}}(k))=0~~\text{ for $k \geq nd-2n$},
$$
which shows the desired claim.
\end{remark}

We are ready to give the proof of Theorem \ref{thm-n(d-2)+1}.

\begin{proof}[Proof of Theorem \ref{thm-n(d-2)+1}]
Recall that $X \subseteq \P^r$ is a non-degenerate smooth projective variety of dimension $n$, codimension $e \geq 2$, and degree $d$. We want to show that
$$
\reg(X) \leq n(d-2)+1.
$$
By \cite{GLP}, we may assume that $n \geq 2$. It is easy to check that $d-e \leq n(d-2)$. Thus, by Theorem \ref{thm-oxreg}, it sufficient to show that $X \subseteq \P^r$ is $n(d-2)$-normal. Suppose that $e \leq n$. It then follows from \cite[Theorem 9]{Noma1} that $X \subseteq \P^r$ is $k$-normal for any integer $k \geq e(d-e)-n$. Since $e(d-e)-n \leq n(d-2)+1$, the assertion holds for the case that $e \leq n$.

It only remains to prove the assertion for the case that $e \geq n+1$. If $e \geq n+2$, then we take a general isomorphic projection to $\P^{2n+1}$. Note that if $X \subseteq \P^{2n+1}$ is $n(d-2)$-normal, then so is $X \subseteq \P^r$. Thus we may assume that $e=n+1$. 
Take a general member $D \in |V_{out}|$. By Corollary \ref{regdpd}, it suffices to show that $D \subseteq \P^{2n+1}$ is $n(d-2)$-normal.
Now, take a general projection 
$$
\pi \colon X \to X' \subseteq \P^{2n}
$$
such that $\pi|_D$ is an isomorphism. 
It is enough to prove that $D \subseteq \P^{2n}$ is $n(d-2)$-normal.
Recall from Remark \ref{rem-dfE} that $X' \subseteq \P^{2n}$ is $n(d-2)$-normal. Thus we only have to prove that the restriction map
$$
H^0(X', \mathcal{O}_{X'}(n(d-2))) \longrightarrow H^0(D, \mathcal{O}_D(n(d-2)))
$$
is surjective. Notice that this surjection follows from the cohomology vanishing
$$
H^1(X', \mathcal{I}_{D|X'}(n(d-2))) = 0,
$$
which we shall show below.

Let $Y \subseteq X'$ be a general hyperplane section of $X' \subseteq \P^{2n}$ so that $Y \subseteq \P^{2n-1}$ is a non-degenerate smooth projective variety of dimension $n-1$, codimension $n$, degree $d$.
By induction on $n$, we may suppose that the theorem holds for smooth projective varieties of dimension $\leq n-1$. Thus $Y \subseteq \P^{2n-1}$ is $((n-1)(d-2)+1)$-regular. By Lemma \ref{lem-hyperplane}, we have 
\begin{equation}\label{h1x'}
H^1(X', \mathcal{O}_{X'}((n-1)(d-2)-1))=0.
\end{equation}
We next consider the following commutative diagram with exact sequences
\[
\xymatrix{
& 0 \ar[d] & 0 \ar[d] & &\\
0 \ar[r] & \mathcal{I}_{D|X'} \ar[r] \ar[d]  & \mathcal{O}_{X'} \ar[r] \ar[d] & \mathcal{O}_D \ar[r] \ar@{=}[d] & 0\\
0 \ar[r] & \pi_* \mathcal{I}_{D|X} \ar[d] \ar[r] & \pi_* \mathcal{O}_X \ar[r] \ar[d] & \mathcal{O}_D \ar[r] & 0 \\
& \mathcal{F} \ar[d] \ar@{=}[r] & \mathcal{F} \ar[d] & &\\
& 0 & 0 &&
}
\]
where $\mathcal{F} \simeq \bigoplus_i \mathcal{O}_{p_i}$ is a coherent sheaf supported at singular points $p_i$ of $X'$. By the cohomology vanishing (\ref{h1x'}), the restriction map
\begin{equation}\label{surj_x'->f}
H^0(X', \pi_*\mathcal{O}_X((n-1)(d-2)-1)) \longrightarrow H^0(X', \mathcal{F})
\end{equation}
is surjective. 
Recall that $D \sim -K_X + (d-n-2)H$, where $H$ is a hyperplane section of $X \subseteq \P^{2n+1}$. Then we have $\mathcal{I}_{D|X}(d-1) \simeq \mathcal{O}_X(K_X + (n+1)H)$. Notice that $K_X + (n+1)H$ is base point free. There is a global section in $H^0(X', \pi_* \mathcal{I}_{D|X}(d-1)) \simeq H^0(X, \mathcal{O}_X(K_X + (n+1)H))$ such that it does not vanish at any singular points $p_i$ of $X'$. 
By considering the surjection of (\ref{surj_x'->f}) and the multiplication map
$$
H^0(X', \pi_* \mathcal{I}_{D|X}(d-1)) \otimes H^0(X', \pi_* \mathcal{O}_X((n-1)(d-2)-1)) \longrightarrow H^0(X', \pi_* \mathcal{I}_{D|X}(n(d-2))),
$$
we see that the restriction map
$$
H^0(X', \pi_*\mathcal{I}_{D|X}(n(d-2))) \longrightarrow H^0(X', \mathcal{F})
$$
is surjective. Since we have 
$H^1(X', \pi_*\mathcal{I}_{D|X}(n(d-2))) = H^1(X, \mathcal{I}_{D|X}(n(d-2))) =0$ by Kodaira vanishing theorem, it follows that 
$$
H^1(X', \mathcal{I}_{D|X'}(n(d-2)))=0.
$$ 
This finishes the proof.
\end{proof}

%%%%%%%%%%%%%%%%%%%%%%%%%%%%%%%%%%%%%%%%%%%%%%%%%%%%%%%%%%
\section{A Castelnuovo-type normality bound by double point divisors}\label{sec-normality}

In this section, we prove Theorems \ref{thm-norm-out} and \ref{thm-norm-inn}, and discuss some related issues. 
We start by recalling a vanishing theorem of Ein-Lazarsfeld in \cite{EL}, which plays an important role in proving Theorems \ref{thm-norm-out} and \ref{thm-norm-inn}.
Let $L$ be a globally generated line bundle on a smooth projective variety $X$, and $W \subseteq H^0(X, L)$ be a base point free subspace. Then we have a short exact sequence
$$0 \longrightarrow M_W \longrightarrow W \otimes \mathcal{O}_X \xrightarrow{e_W} L \longrightarrow 0.$$
If  $W = H^0(X,L)$, we let $M_L := M_W$. Now, let $H$ be a very ample line bundle on $X$, and $B, C$ be nef line bundles on $X$. For integers $f,g$, we write
$L_f = K_X + fH + B \text{ and } N_g = K_X + gH + C$. 
Since we are working in characteristic zero, $\wedge^{q} M_{L_f}$ is a direct summand of $M_{L_f}^{\otimes q}$. Thus the original statements of \cite[Theorem 2.1 and Proposition 3.1]{EL} implies the following.

\begin{theorem}[{\cite[Theorem 2.1 and Proposition 3.1]{EL}}]\label{thm-elvan}
Assume that $(X, H) \neq (\P^n, \mathcal{O}_{\P^n}(1))$.  Then
$H^i (X, \wedge^{q} M_{L_f} \otimes N_g )=0$ for $i \geq 1, f \geq n+1$, and $g \geq n+q-i$.
\end{theorem}

%%%%%%%%%%%%%%%%%%%%%%%%%%%%%%%%%%%%%%%%%%%%%%%%%%%%%%%%%%
\subsection{Outer projection case}\label{subsec-out}

Throughout this subsection, we use the following notations. Let $X \subseteq \P^{r}$ be a non-degenerate smooth projective variety of dimension $n$, codimension $e$, and degree $d$, and $H$ be its hyperplane section. Since  the regularity conjecture holds for smooth curves and hypersurfaces, we assume that $n \geq 2$ and $e \geq 2$. Recall that
$$
D_{out} = -K_X + (d-n-2)H.
$$
is the double point divisor from outer projection.
Let $V:=V_{out}$ be the linear subspace in $H^0(X, \mathcal{O}_X(D_{out}))$ spanned by geometric sections.
By Proposition \ref{prop-dpdoutbpf}, $V$ is a base point free. For an integer $k \geq 0$, let $V_{k}$ be the image of the multiplication map
$$
V \otimes H^0(X, \mathcal{O}_X(k)) \longrightarrow H^0(X, \mathcal{O}_X(D_{out} + kH)),
$$
and $c_k$ be the codimension of $V_k$ in $H^0(X, \mathcal{O}_X(D_{out} + kH))$.

As a consequence of Lemma \ref{outidealsurj}, we prove the following.

\begin{lemma}[A generalized Mumford's lemma for outer projections]\label{mumlemout}
If the multiplication map
$$
V_k \otimes H^0(X, \mathcal O_X(K_X+(n+1+\ell)H))\longrightarrow H^0(X, \mathcal O_X(d-1+k+\ell))
$$
is surjective for some integers $k \geq 0$ and $\ell \in \Z$, then $X \subseteq \P^r$ is $(d-1+k+\ell)$-normal.
\end{lemma}

\begin{proof}
By considering the following commutative diagram
\[
\xymatrix{
V \otimes H^0(\mathcal{O}_X(k)) \otimes H^0(\mathcal O_X(K_X+(n+1+\ell)H)) \ar[d] \ar@{->>}[r] &   V_k \otimes H^0(\mathcal O_X(K_X+(n+1+\ell)H)) \ar@{->>}[d]\\
V \otimes H^0(\mathcal{O}_X(K_X + (n+1+k+\ell)H)) \ar[r] & H^0(\mathcal O_X(d-1+k+\ell)),
}
\]
we see that the bottom map is surjective.
Let $D \in |V|$ and $E \in |K_X + (n+1+k+\ell)H|$ be any effective divisors. Then, by Lemma \ref{outidealsurj}, the effective divisor $D+E \in |(d-1+k+\ell)H|$ is cut out by hypersurface of degree $d-1+k+\ell$ in $\P^r$. This implies that the restriction map
$$
H^0(\P^r, \mathcal{O}_{\P^r}(d-1+k+\ell)) \longrightarrow H^0(X, \mathcal{O}_X(d-1+k+\ell))
$$
is surjective.
\end{proof}

We recall 
$$
m:= \min \{c_k + k \mid  -2K_X+(d-2n-3+k)H \text{ is nef} \}.
$$

\begin{theorem}\label{thm-d-2-normal}
$X \subseteq \P^r$ is $(d-2+\ell)$-normal for all $\ell \ge m$.
\end{theorem}

\begin{proof}
There is an integer $k \geq 0$ such that $m=c_k + k$.
Consider a filtration 
$$
V_{k}:=V^{c_k} \subseteq V^{c_k-1} \subseteq \cdots \subseteq V^0 = H^0(\mathcal{O}_X(D_{out} + kH))$$ 
by subspaces each having codimension one in the next. Note that each $V^i$ is a base point free subspace for $0 \leq i \leq c_k$. We have exact sequences
$$
0 \longrightarrow M_{V^i} \longrightarrow V^i\otimes \mathcal{O}_X \longrightarrow \mathcal O_X(D_{out}+kH) \longrightarrow 0
$$
$$
0 \longrightarrow M_{V^{i+1}} \longrightarrow M_{V^i} \longrightarrow \mathcal{O}_X \longrightarrow 0.
$$
We first show that
$$
V^{c_k} \otimes H^0(X, \mathcal{O}_X(K_X + (n+c)H)) \longrightarrow H^0(X, \mathcal{O}_X(d-2+k+c))
$$
is surjective for all $c \geq c_k$. For this purpose, it suffices to establish that
\begin{equation}\label{d+mcovan}
H^1(X, M_{V^{c_k}} \otimes \mathcal{O}_X(K_X + (n+c)H)) = 0 ~~\text{ for ~~all} ~~ c \geq c_k.
\end{equation}
Note that $D_{out}+kH = K_X + (n+1)H -2K_X + (d-2n-3+k)H$ and $-2K_X+(d-2n-3+k)H$ are assumed to be nef. By Theorem \ref{thm-elvan}, we have
$$
H^j(X, \wedge^{c_k+j}M_{V^0}\otimes \mathcal{O}_X(K_X + nH + c H)) = 0 ~~~~\text{for ~~all}~~ j\ge 1, ~~c \ge c_k.
$$
Consider the following short exact sequence
$$
0 \longrightarrow \wedge^p M_{V^{i+1}} \longrightarrow \wedge^p M_{V^i} \longrightarrow \wedge^{p-1} M_{V^{i+1}} \longrightarrow 0.
$$
Twisting by $\mathcal{O}_X (K_X + nH + c H)$ and taking cohomology sequence, it is easy to check that
$$
H^j (X, \wedge^{c_1+j}M_{V^1}\otimes \mathcal{O}_X(K_X + nH + c H)) = 0 ~~~~\text{for ~~all}~~ j\ge 1, ~~c \ge c_k, ~~c-1\ge c_1\ge 0.
$$
By an induction on $i$ and similar arguments, we can also show that
$$
H^j (X, \wedge^{c_i+j}M_{V^i}\otimes \mathcal{O}_X(K_X + nH + c H)) = 0 ~~~~\text{for ~~all}~~ j\ge 1, ~~c \ge c_k, ~~c-i\ge c_i\ge 0. $$
In particular, we get the cohomology vanishing (\ref{d+mcovan}).
This implies that the multiplication map
$$
V_k \otimes H^0(X, \mathcal{O}_X(K_X + (n+c)H)) \to H^0(X, \mathcal{O}_X(d-2+k+c))
$$
is surjective for all $c \geq c_k$. Now, the assertion follows from  Lemma \ref{mumlemout}.
\end{proof}

To prove Theorem \ref{thm-norm-out}, we may assume that $X \subseteq \P^r$ is non-degenerate and $n, e \geq 2$. Then Theorem \ref{thm-norm-out} follows from Theorem \ref{thm-oxreg} and Theorem \ref{thm-d-2-normal}.

In view of Theorem \ref{thm-d-2-normal}, it is natural to ask when the divisor $-2K_X+(d-2n-3+k)H$ appeared in the definition of $m$ is nef. If $H$ is sufficiently positive (e.g., $H=\ell A$, where $A$ is a very ample divisor and $\ell \geq 2$ is an integer), then $-2K_X + (d-2n-3)H$ is already nef. In general, we have the following.

\begin{proposition}\label{d-1nef}
$-2K_X+(d-2n-3+k)H$ is nef for $k \geq d-1$.
\end{proposition}

\begin{proof}
We write $-2K_X+(d-2n-3+k)H = 2D_{out} + (k-d+1)H$. Since $D_{out}$ is base point free by Proposition \ref{prop-dpdoutbpf}, the assertion follows.
\end{proof}

To give a bound for $m$ in terms of $d,n,e$, it is necessary to control the codimension $c_k$.
One approach is to use the hyperplane section method.  For this purpose, we fix some notations. 
Let $Y \subseteq \P^{r-1}$ be a general hyperplane section of $X \subseteq \P^r$, which is a non-degenerate smooth projective variety of dimension $n-1$, codimension $e$, and degree $d$.
Note that
$$
D_{out}|_Y= -(K_X + H)|_Y + (d-(n-1)-2)H|_Y = -K_Y + (d-(n-1)-2)H|_Y.
$$
By abuse of notation, we write $D_{out}$ for the double point divisor from outer projection of $Y \subseteq \P^{r-1}$ and $H$ for general hyperplane section of $Y \subseteq \P^{r-1}$.
Let $V_X$ and $V_Y$ be subspaces of $H^0(X, \mathcal{O}_X(D_{out}))$ and $H^0(Y, \mathcal{O}_Y(D_{out}))$ spanned by geometric sections, respectively, and $c_k^X$ and $c_k^Y$ be the codimensions of the images of the maps $V_X \otimes H^0(\mathcal{O}_X(k)) \rightarrow H^0(X, \mathcal{O}_X(D_{out} + kH))$ and $V_Y \otimes H^0(\mathcal{O}_Y(k)) \rightarrow H^0(Y, \mathcal{O}_Y(D_{out}+kH))$, respectively, for any integer $k \geq 0$.

Note that there is a natural injective map $|V_Y| \hookrightarrow |V_X|$, which induces an injective map 
$V_Y \hookrightarrow V_X$.
We then have the following: 
\begin{enumerate}[$(1)$]
 \item If $H^1(X, \mathcal{O}_X(k))=H^1(X, \mathcal{O}_X(k+1))=H^1(X, \mathcal{O}_X(D_{out} + kH))=H^2(X, \mathcal{O}_X(k))=H^1(Y, \mathcal{O}_Y(k+1))=0$ (these conditions are satisfied if $k \geq d-1$), then we have
$$
c_{k+1}^X \leq c_k^X + c_{k+1}^Y.
$$
 \item If $H^1(X, \mathcal{O}_X( \ell ))=H^1(X, \mathcal{O}_X(D_{out}+\ell H))=H^2(X, \mathcal{O}_X(\ell))=H^1(Y, \mathcal{O}_Y(\ell+1))=0$ for $\ell \geq k$ and $M_{V_X} \otimes \mathcal{O}_Y$ is $(k+1)$-regular (these conditions are satisfied if $k \geq d-1$ and $c_{k}^Y=0$), then we have
$$
c_{k+1}^X < c_k^X~\text{ provided that $c_k^X \neq 0$.}
$$
\end{enumerate}
The proof is straightforward by the standard vector bundle techniques, so we leave the details to the interested readers.

Suppose now that $c_{\ell}^Y=0$ for some integer $\ell \geq d-1$. From now on, we write $c_k=c_k^X$. Then, for $k \geq \ell$, we have
$$
c_{k+1} +k+1 \leq c_k + k
$$ 
provided that $c_k \neq 0$. Since $-2K_X + (d-2n-3+k)H$ is nef for $k \geq d-1$ by Proposition \ref{d-1nef}, calculating $m$ is almost equivalent to finding the minimum $k$ such that $c_k=0$ for $k \geq d-1$.
In this respect, we show the following.

\begin{proposition}\label{n(d-3)zero}
We have $c_{k}=0$ for $k \geq n(d-3)$.
\end{proposition}

\begin{proof}
We can take a subspace $W$ of $V$ generated by $n+1$ general geometric sections so that $W$ is also base point free. We have the following exact sequence
$$
0 \longrightarrow M_W  \longrightarrow W \otimes \mathcal{O}_X  \longrightarrow \mathcal{O}_X(D_{out})  \longrightarrow 0.
$$
Since $M_W$ is a rank $n$ vector bundle, $M_W \simeq  \wedge^{n-1} M_W^* \otimes \mathcal{O}_X(-D_{out})$. 
We shall show that
$$
H^1(X, \wedge^{n-1} M_W^* \otimes \mathcal{O}_X(-D_{out} + kH))=0~~\text{ for $k \geq n(d-3)$}.
$$
Now, since $(X, \mathcal{O}(1)) \neq (\P^n, \mathcal{O}_{\P^n}(1))$, it follows from the main result of \cite{E} that $K_X + nH$ is base point free. Notice that $k \geq n(d-3) \geq jd-2j-n$ for $1 \leq j \leq n-1$. We write
$$
-jD_{out} + kH = K_X + (j-1)(K_X + nH) + (k-jd + 2j + n)H.
$$
Then, by Kodaira vanishing theorem, we have
$$
H^i(X, \mathcal{O}_X(-jD_{out}+kH))=0~~ \text{ for $i > 0$},
$$
For each $1 \leq i \leq n$, from the following exact sequence
$$
0  \rightarrow \wedge^{n-i}M_W^* \otimes \mathcal{O}_X(-D_{out})  \rightarrow \wedge^{n-i+1}W^* \otimes \mathcal{O}_X \rightarrow  \wedge^{n-i+1}M_W^*  \rightarrow 0,
$$
we obtain
$$
H^{i-1}(\wedge^{n-(i-1)}M_W^* \otimes \mathcal{O}_X(-(i-1)D_{out}+kH))=H^i(\wedge^{n-i}M_W^* \otimes \mathcal{O}_X(-iD_{out}+kH))~~\text{ for $2 \leq i \leq n$.}
$$
This implies that
$$
H^1(X, \wedge^{n-1} M_W^* \otimes \mathcal{O}_X(-D_{out} + kH)) = H^n(X, \mathcal{O}_X(-nD_{out}+kH))=0.
$$
This proves the assertion.
\end{proof}

\begin{remark}\label{rem-expectation}
There are only a few cases such that $n(d-3) < d-1$, and the regularity conjecture holds for all such cases. Thus we may assume that $n(d-3) \geq d-1$. Then Propositions \ref{d-1nef} and \ref{n(d-3)zero} imply that $m \leq n(d-3)$. By Theorem \ref{thm-norm-out}, we obtain 
$$
\reg(X) \leq (n+1)(d-3)+2.
$$
which is unfortunately weaker than Theorem \ref{thm-n(d-2)+1}. In the proof of Proposition  \ref{n(d-3)zero}, only $n+1$ general geometric sections in $V$ are used. We expect that a systematic approach to use of \emph{all} geometric sections in $V$ would lead us to a better bound for $m$ as well as $\reg(X)$.
\end{remark}

%%%%%%%%%%%%%%%%%%%%%%%%%%%%%%%%%%%%%%%%%%%%%%%%%%%%%%%%%%
\subsection{Inner projection case}\label{subsec-inn}

The cohomological method in Subsection \ref{subsec-out} can be directly generalized to the inner projection case under suitable assumptions. Throughout this subsection,  we use the following notations. Let $X \subseteq \P^{r}$ be a non-degenerate smooth projective variety of dimension $n$, codimension $e$, and degree $d$, and $H$ be its hyperplane section. As in Subsection \ref{subsec-out}, we may assume that $n, e \geq 2$. 
We also need the following conditions:
\begin{enumerate}[$(1)$]
\item  Assume that $X \subseteq \P^r$ is neither a scroll over a curve, the second Veronese surface in $\P^5$, nor a Roth variety. Then we can consider the double point divisor from inner projection
$$
D_{inn} =-K_X + (d-n-e-1)H,
$$
which is semiample by Theorem \ref{thm-dinnsemiample}. Let $V':=V_{inn}$ be the linear subspace of $H^0(\mathcal{O}_X(D_{inn}))$ spanned by geometric sections. 
\item Assume that $V'$ is base point free. Recall that $\Bs(|V'|)\subseteq \mathcal{C}(X)$. Thus if $\mathcal{C}(X) = \emptyset$, then $V'$ is base point free (see \cite[Corollary 3]{Noma0} for examples with $\mathcal{C}(X) = \emptyset$).
\item Assume that the multiplication map
$$
W_i \otimes H^0(X, \mathcal{O}_X(K_X+(n-1)H)) \longrightarrow H^0(X, \mathcal{O}_X(K_X + (n-1+i)H))
$$ 
is surjective for $i=1,2$, where $W_k$ is the image of the map $H^0(\P^r, \mathcal{O}_{\P^r}(k)) \to H^0(X, \mathcal{O}_X(k))$ for any integer $k$. Since $\mathcal{O}_X(K_X+(n+1)H)$ is $0$-regular with respect to $\mathcal{O}_X(1)$, it follows that the multiplication map
$$
W_k \otimes H^0(X, \mathcal{O}_X(K_X+ (n-1)H)) \longrightarrow H^0(X, \mathcal{O}_X(K_X + (n-1+k)H))
$$
is surjective for any integer $k \geq 0$.
\end{enumerate}

We first prove the following, which is a counterpart of Lemma \ref{outidealsurj}.

\begin{lemma}\label{surjidealinn}
Let $D \in |V'|$ be an effective divisor on $X$.
Then the natural restriction map
$$
H^0(\P^r, \mathcal{I}_{D|\P^r}(d-e-2+\ell)) \longrightarrow H^0(X, \mathcal{I}_{D|X}(d-e-2+\ell))
$$
is surjective for any integer $\ell \leq 0$.
\end{lemma}

\begin{proof}
As in Lemma \ref{outidealsurj}, it is sufficient to show the assertion for a geometric divisor $D=D_{inn}(\Lambda)$  associated to a general inner projection $\pi=\pi_{\Lambda} \colon X \dashrightarrow \overline{X} \subseteq \P^{n+1}$ centered at a general linear subspace $\Lambda \subseteq \P^r$ meeting $X$ at general $(e-1)$ points on $X$. Let $x_1, \ldots, x_{e-1}$ be general points on $X$ such that $X \cap \Lambda = \{ x_1, \ldots, x_{e-1} \}$, and $\sigma \colon \widetilde{X} \rightarrow X$ be the blow-up at $x_1, \ldots, x_{e-1}$ with exceptional divisors $E_1, \ldots, E_{e-1}$. Put $E:=E_1 + \cdots + E_{e-1}$.
Note that the birational morphism $\widetilde{\pi}:=\sigma \circ \pi \colon \widetilde{X} \rightarrow \overline{X}$ is a resolution of singularities and has no exceptional divisor by our assumption.
Let $\overline{H}$ be a general hyperplane section of $\overline{X} \subseteq \P^{n+1}$. Then we have $\widetilde{\pi}^* \overline{H} = \sigma^* H - E$. Let $\widetilde{D}$ be the non-isomorphic locus of $\widetilde{\pi}$ so that $\sigma(\widetilde{D})=D_{inn}(\Lambda)$.  By the birational double point formula (\cite[Lemma 10.2.8]{positivity}), we have
$$
\widetilde{D} \sim -K_{\widetilde{X}} + (d-n-e-1)\widetilde{\pi}^* \overline{H} \sim \sigma^*(D_{inn}) + (-d+e+2)E.
$$
The image $\overline{D}:=\overline{\widetilde{\pi}(\widetilde{D})}$ is a Weil divisor on $\overline{X}$. We then obtain the following diagram:
\[
 \xymatrix{
\widetilde{D}  \ar@{^{(}->}[r] \ar_-{\widetilde{\pi}|_{\widetilde{D}}}[d]  & \widetilde{X} \ar^{\sigma}[r] \ar_-{\widetilde{\pi}}[rd] & X \ar@{.>}[d]^-{\pi} \ar@{^{(}->}[r] & \P^{r} \ar@{.>}[d]^-{\pi}\\
\overline{D}  \ar@{^{(}->}[rr] & & \overline{X} \ar@{^{(}->}[r] & \P^{n+1}.
}
\]
Note that $\mathcal{I}_{\overline{D}|\P^{n+1}}=\text{adj}(\P^{n+1}, \overline{D})$ is the adjoint ideal and  $\mathcal{I}_{\overline{D}|\overline{X}} \simeq \widetilde{\pi}_* \mathcal{I}_{\widetilde{D}|\widetilde{X}}= \widetilde{\pi}_{*}\mathcal{O}_{\widetilde{X}} (-\widetilde{D})$ (see \cite[Proposition 9.3.48]{positivity}).
Consider the following short exact sequence
$$
0 \longrightarrow \mathcal{I}_{\overline{X}|\P^{n+1}} \longrightarrow \mathcal{I}_{\overline{D}|\P^{n+1}} \longrightarrow \mathcal{I}_{\overline{D}|\overline{X}} \longrightarrow 0.
$$
For any integer $\ell$, we have
$$
(d-e-2+\ell)\widetilde{\pi}^*\overline{H} - \widetilde{D} \sim K_{\widetilde{X}} + (n-1+\ell)\widetilde{\pi}^* \overline{H} \sim \sigma^*(K_X + (n-1+\ell)H) - \ell E.
$$
Thus, for any integer $\ell \leq 0$, we obtain
$$
H^0(\overline{X}, \mathcal{I}_{\overline{D}|\overline{X}}(d-e-2+\ell)) = H^0(X, \mathcal{O}_X(K_X + (n-1+\ell)H)) = H^0(X, \mathcal{I}_{D|X}(d-e-2+\ell)).
$$
We have the following commutative diagram:
\[\xymatrix{
H^0(\P^{n+1}, \mathcal{I}_{\overline{D}|\P^{n+1}}(d-e-2+\ell)) \ar@{^{(}->}[d] \ar@{->>}[r] &  H^0(\overline{X}, \mathcal{I}_{\overline{D}|\overline{X}}(d-e-2+\ell)) \ar@{=}[d] \\
H^0(\P^r, \mathcal{I}_{D|\P^r}(d-e-2+\ell)) \ar[r]  & H^0(X, \mathcal{I}_{D|X}(d-e-2+\ell)).
}\]
Since $\mathcal{I}_{\overline{X}|\P^{n+1}} \simeq \mathcal{O}_{\P^{n+1}}(-(d-e+1))$ and $H^1(\P^{n+1}, \mathcal{O}_{\P^{n+1}}(\ell))=0$, it follows that the upper horizontal map is surjective. Thus the lower horizontal map is also surjective.
\end{proof}

\begin{remark}
Unlike Lemma \ref{outidealsurj}, we need to assume that $\ell \leq 0$ in Lemma \ref{surjidealinn}. If $\ell \geq 1$, then we only have
$$
H^0(\overline{X}, \mathcal{I}_{\overline{D}|\overline{X}}(d-e-2+\ell)) \subseteq H^0(X, \mathcal{I}_{D|X}(d-e-2+\ell)),
$$
so we cannot deduce the surjectivity of the map 
$$
H^0(\P^r, \mathcal{I}_{D|\P^r}(d-e-2+\ell)) \longrightarrow H^0(X, \mathcal{I}_{D|X}(d-e-2+\ell)).
$$ 
For instance, consider a projected second Veronese embedding $X \subseteq \P^4$. We have $d=4$ and $e=2$. Let $A=\frac{1}{2}H$ be the ample generator of $\Pic(X)$. Then $D_{inn} \sim A$. 
For any  geometric divisor $D \in |D_{inn}|$, we see that $H^0(\P^4, \mathcal{I}_{D|\P^4}(1)) \to H^0(X, \mathcal{I}_{D|X}(1))$ is not surjective.
\end{remark}

Let $V_k'$ be the image of of the multiplication map
$$
V' \otimes H^0(X, \mathcal{O}_X(k)) \longrightarrow H^0(X, \mathcal{O}_X(D_{inn} + kH))
$$
for any integer $k \geq 0$, and $c_k'$ be the codimension of $V_k'$ in $H^0(X, \mathcal{O}_X(D_{inn} + kH))$.

\begin{lemma}[A generalized Mumford's lemma for inner projection]\label{mumleminn}
If the multiplication map
$$
V_k' \otimes H^0(X, \mathcal{O}_X(K_X + (n-1)H)) \longrightarrow H^0(X, \mathcal{O}_X(d-e-2+k))
$$
is surjective for some integer $k \geq 0$, then $X \subseteq \P^r$ is $(d-e-2+k)$-normal.
\end{lemma}

\begin{proof}
By considering the following commutative diagram
\[
\xymatrix{
V' \otimes H^0(\mathcal{O}_X(k)) \otimes H^0(\mathcal O_X(K_X+(n-1)H)) \ar[d] \ar@{->>}[r] &   V_k' \otimes H^0(\mathcal O_X(K_X+(n-1)H)) \ar@{->>}[d]\\
V' \otimes H^0(\mathcal{O}_X(K_X + (n-1+k)H)) \ar[r] & H^0(\mathcal O_X(d-e-2+k)),
}
\]
we see that the bottom map is surjective. Recall that the multiplication map
$$
W_k \otimes H^0(X, \mathcal{O}_X(K_X+ (n-1)H)) \longrightarrow H^0(X, \mathcal{O}_X(K_X + (n-1+k)H))
$$
is surjective. By Lemma \ref{surjidealinn}, we have the map
$$
V' \otimes   H^0(X, \mathcal{O}_X(K_X + (n-1)H)) \longrightarrow W_{d-e-2} \subseteq H^0(X, \mathcal{O}_X(d-e-2)).
$$
By considering the following commutative diagram
\[
\xymatrix{
V' \otimes W_k \otimes H^0(\mathcal O_X(K_X+(n-1)H)) \ar[d] \ar@{->>}[r] &   V' \otimes H^0(\mathcal O_X(K_X+(n-1+k)H)) \ar@{->>}[d]\\
W_{d-e-2} \otimes W_k \ar[r] & H^0(\mathcal O_X(d-e-2+k)),
}
\]
we see that the bottom map is surjective, and hence, the assertion follows.
\end{proof}

As in the outer projection case, we set
$$
m':= \min \{c_k' + k \mid  -2K_X+(d-2n-e-2+k)H \text{ is nef} \}.
$$
By the same argument in the proof of Theorem \ref{thm-d-2-normal} using Lemma \ref{mumleminn}, we can prove the following theorem. We leave the details to the interested readers.

\begin{theorem}\label{thm_inn}
$X \subseteq \P^r$ is $(d-e-2+\ell)$-normal for all $\ell \ge m'$.
\end{theorem}

Now, Theorem \ref{thm-norm-inn} follows from Theorem \ref{thm-oxreg} and Theorem \ref{thm_inn}.

$ $
%%%%%%%%%%%%%%%%%%%%%%%%%%%%%%%%%%%%%%%%%%%%%%%%%%%%%%%%%%
%BIBLIOGRAFIA

\end{document}